\documentclass[11pt,english,oneside]{amsart}
\usepackage[T1]{fontenc}
\usepackage[latin9]{inputenc}
\usepackage{geometry}
\usepackage{amssymb}
\usepackage{color}

\usepackage{graphicx,color}
\usepackage{amsmath, amssymb, graphics}

\usepackage{graphicx}
\usepackage{placeins}

\makeatletter
\numberwithin{equation}{section} 
\numberwithin{figure}{section} 
  \@ifundefined{theoremstyle}{\usepackage{amsthm}}{}
  \theoremstyle{plain}
  \newtheorem{thm}{Theorem}[section]
  \theoremstyle{plain}
  
  \theoremstyle{plain}
  \newtheorem{prop}[thm]{Proposition}
  \theoremstyle{Remark}
  \newtheorem{rem}[thm]{Remark}
  \theoremstyle{remark}
  
  \theoremstyle{plain}
  \newtheorem{lem}[thm]{Lemma}
  \newtheorem{mydef}{Definition}

\usepackage{geometry}

\def\bfR#1{{\bf R}^#1}

\def\com#1{ \hbox{#1}}

\smallskip
\def\<{{\langle }}
\def\>{{\rangle }}

\makeatother

\usepackage{babel}

\def\bfR#1{{\bf R}^#1}

\def\com#1{ \quad\hbox{#1}\quad}

\smallskip
\def\<{{\langle }}
\def\>{{\rangle }}

\makeatother

\begin{document}

\title[helicoidal rotating drops]{Rotating Drops with Helicoidal Symmetry}

\author{Bennett Palmer and  Oscar M. Perdomo}

\date{\today}

\curraddr{O. Perdomo\\
Department of Mathematics\\
Central Connecticut State University\\
New Britain, CT 06050 USA\\
e-mail: perdomoosm@ccsu.edu
}
\curraddr{B. Palmer\\
Department of Mathematics\\
Idaho  State University\\
Pocatello, ID, 83209 USA\\
e-mail: palmbenn@isu.edu
}


\subjclass[2000]{ 53C42, 53C10}

\maketitle

\begin{abstract}

We  consider helicoidal immersions in $\bfR{3}$ whose axis of symmetry is the $z$-axis that are solutions of the equation  $2 H=\Lambda_0-a \frac{R^2}{2}$ where $H$ is the mean curvature of the surface, $R$ is the distance form the point in the surface to the $z$-axis and $a$ is a real number. We refer to these surfaces as helicoidal rotating drops. We  prove the existence of properly immersed solutions that contain the $z$-axis. We also show the existence of several families of embedded examples. We  describe the set of possible solutions and we  show that most of these solutions are not properly immerse and are dense in the region bounded by two concentric cylinders. 
We show that all properly immersed solutions, besides being invariant under a one parameter helicoidal group, they are invariant under a cyclic group of rotations of the variables $x$ and $y$.

The second variation of energy for the volume constrained problem with Dirichlet boundary conditions is also studied.
\end{abstract}
\begin{figure}[h]\label{bounds}
\centerline{ \includegraphics[width=5cm,height=5cm]{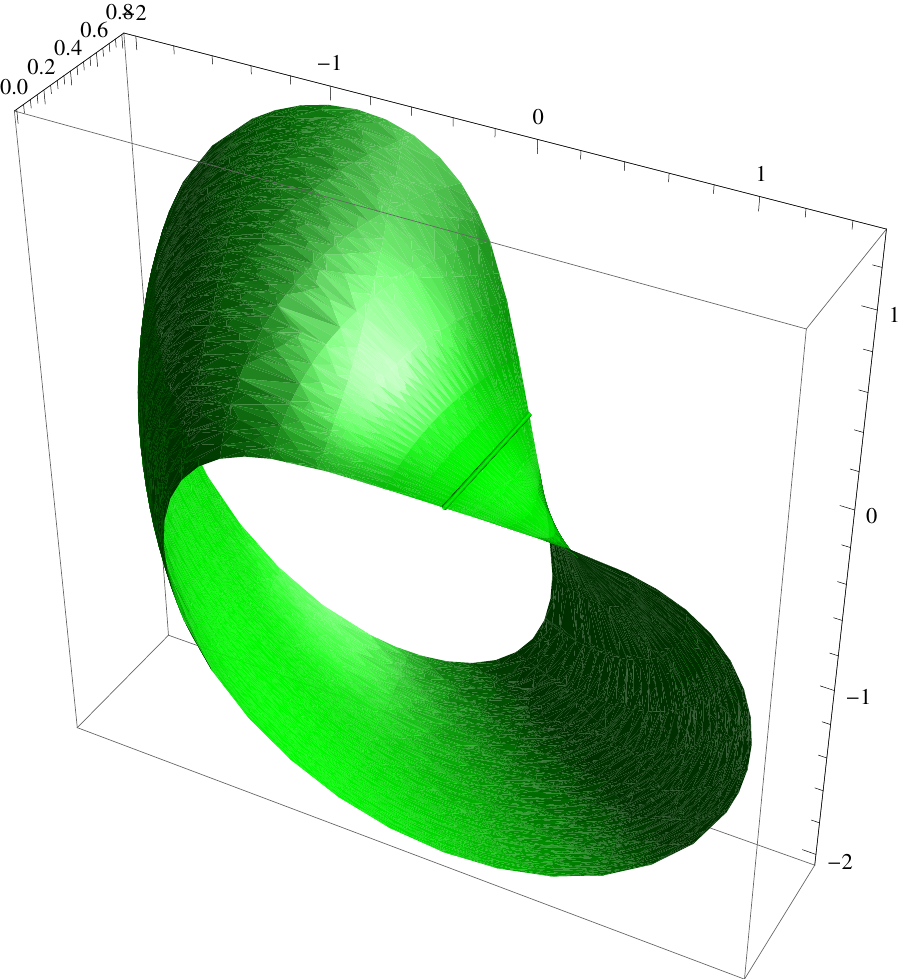}}

\end{figure}

\section{Introduction}

In this paper we  study the equilibrium shape of a rotating liquid drop or liquid film, which is invariant under a helicoidal motion of the three dimensional Euclidean space.
The subject of rotating drops has been studied by many authors, including  Chandrasekhar \cite{SC}, Brown and Scriven \cite{BS}, Solonikov \cite{S} and many others. 
Our main objective here is to use a new construction, recently developed by the second author \cite{P1}, to construct an abundant supply of examples. This construction is closely related to the Delaunay's classical construction of the axially symmetric constant mean curvature surfaces whose generating curves are produced by rolling a conic section. A special case of the type of surface which we study here, occurs when the rotating drop is a cylinder over a plane curve. We treat that case in detail in \cite{PP}.

If a rigid object is moved from one position in space to another, this repositioning can be realized via a helicoidal motion of ${\bf R}^3$. If then, this motion is successively repeated,
one arrives at a configuration which is invariant under a helicoidal motion. The simple idea that helicoidal motions give all repeated motions of a rigid object,  known in the physical sciences as  Pauling's Theorem,  is behind many of the occurrences of helicoidal symmetry in nature, since it allows for extensive growth with a minimal amount of information.  

We will consider the equilibrium shape of a liquid drop rotating with a constant angular velocity $\Omega$ about a vertical axis. The surface of the drop, which we denote by $\Sigma$, is represented as a smooth surface. 
The bulk of the drop is assumed to be occupied by an imcomressible liquid of a constant mass density $\rho_1$ while the drop is surrounded by a fluid of constant mass density $\rho_2$.
Since the drop is liquid, its free surface energy is proportional to its surface area ${\mathcal A}$ and we take the constant of proportionality to be one. 
 The rotation contributes a  second energy term of the form $-\Omega^2 \Delta {\mathcal I}$, where $\Delta {\mathcal I}$ is  difference of moments of inertia about the vertical axis,
 $$\Delta {\mathcal I}:=(\rho_1-\rho_2) \int_UR^2\:dv\:. $$
 
  This term represents
 twice the rotational kinetic energy. 
 
 The total energy is thus of the form
 \begin{equation}
 \label{E}
 {\mathcal E}:={\mathcal A}-\frac{\Omega^2}{2}{\Delta \mathcal I}+\Lambda_0{\mathcal V}\:,
 \end{equation}
 where ${\mathcal V}$ denotes the volume of the drop and $\Lambda_0$ is a Lagrange multiplier. Let $\Delta \rho:=\rho_1-\rho_2$, then by introducing a constant  $a :=(\Delta \rho) \Omega^2$, we can write the functional in the form
  \begin{equation}
 \label{E}
{\mathcal E}_{a, \Lambda_0}= {\mathcal A}-\frac{a}{2} \int_U R^2\:dV+\Lambda_0 {\mathcal V}\:,\nonumber
 \end{equation}
where $U$ is the three dimensional region occupied by the bulk of the drop, and $R:=\sqrt{x^2+y^2}$.  

Since we want to consider both embedded and immersed surfaces, we will precisely define the last two terms in the energy in the following way. First define vector fields on ${\bf R}^3$ by
$$W=\nabla ' R^4/16=\frac{R^2}{4}(x_1, x_2, 0)\:, W_0=\nabla ' R^2/4=(1/2)(x_1, x_2, 0)\:.$$
If $\nabla'\cdot$ denotes the divergence operator on ${\bf R}^3$, then it is easily checked that  $\nabla ' \cdot W_0=1$ and $\nabla'\cdot W=R^2$ hold. 
We, then define
$${\mathcal V}:=\int_\Sigma W_0\cdot \nu\:d\Sigma\:,\quad \int_UR^2\:dv:=\int_\Sigma W\cdot \nu\:d\Sigma\:.$$
The definitions are valid so long as $\Sigma$ is immersed and oriented.

We will next derive the first variation of the functional given above. Let $X_\epsilon=X+\epsilon (\psi \nu +T)+...$ be a variation of $X$, where $\psi$ is a smooth function, $\nu$ is the unit normal to the surface and $T$
is a tangent vector field along $\Sigma$.  The first variation formula for the area gives,
\begin{eqnarray*}
\delta{\mathcal A}&=&-\int_\Sigma 2H \psi\:d\Sigma+\oint_{\partial \Sigma}  T\cdot n\:ds\\
&=&-\int_\Sigma 2H \psi\:d\Sigma+\oint_{\partial \Sigma}dX\times \nu \cdot T\:.
\end{eqnarray*}

We will show in the appendix that 
\begin{equation}
\label{dR} \delta \int_\Omega R^2\:dV=\int_\Sigma \psi R^2\:d\Sigma +\oint_{\partial \Sigma}dX\times W\cdot \delta X\:,\end{equation}
where $W$ is a vector field satisfying $\nabla '\cdot W  =R^2$ on ${\bf R}^3$,
and it is well known that the first variation of volume is

\begin{equation}
\label{dV}\delta {\mathcal V}=\int_\Sigma \psi \:d\Sigma +\oint _{\partial \Sigma} dX\times W_0 \cdot \delta X\:.\end{equation}

By combining the last three formulas, we arrive at
\begin{eqnarray}
\label{fv}
\delta {\mathcal E}_{a, \Lambda_0}&=&\int_\Sigma (-2H-\frac{a}{2}R^2+\Lambda_0)\psi\:d\Sigma\nonumber \\
&&+\oint_{\partial \Sigma_1} dX\times (\nu-\frac{a}{2} W+\Lambda_0 W_0)\cdot \delta X
\end{eqnarray}

Regardless of the boundary conditions, a necessary condition for an equilibrium is that in the interior of $\Sigma$, there holds
\begin{equation}
\label{EL}
2H=-\frac{a}{2}R^2+\Lambda_0\:.
\end{equation}

If we assume that the surface has free boundary contained in a supporting surface $S$ having outward unit normal $N$, then the admissible variations must satisfy
the condition $\delta X\cdot N\equiv 0$ on $\partial \Sigma$. In order that the boundary integral in (\ref{fv}) vanishes for all admissible variations, we must have that $dX\times \nu$
parallel to $N$ along the boundary which means that the surface $\Sigma$ meets the supporting surface $S$ in a right angle.

We now assume that an equilibrium surface $\Sigma$, i.e. a surface satisfying (\ref{EL}) is invariant under a helicoidal motion:
\begin{equation}
\label{HM} (x_1+ix_2, x_3)\mapsto (e^{-i\omega t}(x_1+ix_2), x_3+t)\:,\end{equation}
and we will  derive a conservation law which characterizes the equilibrium surfaces. We do not assume that the angular velocity $\omega$  which determines the pitch of the helicoidal surface is the same as the angular velocity $\Omega$ appearing above,

Let $\Sigma_1$ denote the compact region in $\Sigma$ bounded on the sides by two integral curves $C_1$ and $C_2$ of the Killing field ${\mathcal K}(X):=-\omega E_3\times X+E_3$
and bounded below and above by the horizontal planes $x_3=0$ and $x_3=2\pi/\omega$. Then, $\Sigma_1$ is a compact surface with oriented boundary $C_1+\alpha_1 -C_2
-\alpha_2$ where $\alpha_1$ and $\alpha_2$ are congruent arcs in the planes $x_3=2\pi / \omega$, $x_3=0$, respectively. By the calculations in the appendix,  we have, using (\ref {EL})
\begin{equation}
\label{de2}
\delta {\mathcal E}[\Sigma_0]=\oint_{\partial \Sigma_1} dX\times (\nu-\frac{a}{2} W+\Lambda_0 W_0)\cdot \delta X\:.
\end{equation}

If we take the variation with $\delta X=E_3$ then, since $E_3$ generates a translation, the first variation will vanish. Consequently, we obtain
\begin{equation}
\label{bint} 0=\oint_{\partial \Sigma_1}  dX\times (\nu-\frac{a}{2} \nabla'\frac{R^4}{16}+\Lambda_0 \nabla '\frac{R^2}{4})\cdot E_3\:.\end{equation}

Note that the integration over the 1-chain $\alpha_1-\alpha_2$ yields zero since the two arcs are congruent and are traversed in opposite directions. On $C_i$, $i=1,2$, we 
have $dX/dt=(-1)^{i+1} {\mathcal K}$.
\begin{eqnarray*}
 dX\times (\nu-\frac{a}{2} \nabla'\frac{R^4}{16}+\Lambda_0 \nabla '\frac{R^2}{4})\cdot E_3&=&(-1)^{i+1}  (-\omega E_3\times X+E_3)\times  (\nu-\frac{a}{2} \nabla'\frac{R^4}{16}+\Lambda_0 \nabla '\frac{R^2}{4})\cdot E_3\:dt\\
 &=&(-1)^{i}  (-\omega E_3\times X+E_3)\times E_3\cdot (\nu-\frac{a}{2} \nabla'\frac{R^4}{16}+\Lambda_0 \nabla '\frac{R^2}{4})\:dt\\
  &=&(-1)^{i+1} \omega (x_1, x_2, 0)\cdot (\nu-\frac{a}{8}R^2(x_1, x_2, 0)+\frac{\Lambda_0}{2} (x_1, x_2, 0))\:dt\\
    &=&(-1)^{i+1}\omega\bigl( (Q-x_3\nu_3)-\frac{a}{8}R^4+\frac{\Lambda_0R^2}{2}\bigr)\:dt\\
 \end{eqnarray*} 
where $Q=X\cdot \nu$ is the support function of the surface. Setting ${\hat Q}:=Q-x_3\nu_3$,  We can conclude from this that the integral
$$\int_{C_i} \bigl( {\hat Q}-\frac{a}{8}R^4+\frac{\Lambda_0R^2}{2}\bigr)\:dt\:\:,$$
is independent of $i$. Also, it is easily checked that the integrand is, in fact,  constant on each helix $C_i$ and we obtain the result that
\begin{equation}
\label{con}
2{\hat Q}+\Lambda_0R^2-a\frac{R^4}{4}\equiv {\rm constant}\:.\end{equation}
\begin{prop}
\label{Q}
Let $\Sigma$ be a helicoidal surface. A necessary and sufficient condition that $\Sigma$ is a critical point for the functional ${\mathcal E}_{a, \Lambda_0}$ is that (\ref{con}) holds.
\end{prop}
\begin{proof}
The necessity was shown above, we now show that the condition is sufficient. We can assume that the helicoidal symmetry group of the surface fixes the vertical axis.

Any helicoidal surface arises as the orbit of a planar `` generating curve'' $\alpha$ under a helicoidal motion. We let  $s$ be the arc length coordinate of $\alpha$ and we let $t$ denote a coordinate for the helices which are the orbits of points in $\alpha$. Local calculations which can be found in \cite{P1}, show that the mean curvature $H$ and the third component of the normal $\nu_3$ are functions of $s$ alone. Also, it is clear that the function $R^2$ only depends on $s$. 

It is easy to see that if $\nu_3$ vanishes on any arc of $\alpha$, then this arc is necessarily circular. It is clear that the orbit of a circular arc  satisfying $\ref{con}$  is a critical surface for the functional ${\mathcal E}_{a, \Lambda_0}$.  Now consider a connected arc $\eta \subset \alpha$ on which  (say)  $\nu_3>0$ holds  almost everywhere. If (\ref{EL}) does not hold  on $\alpha$, we can assume, by replacing $\alpha$ with a sub-arc if necessary, that $-2H-aR^2/2+\Lambda_0>0$ holds almost everywhere on $\alpha$ also.  Let $\Sigma_1$ denote the compact domain consisting of the orbit of the arc $\alpha$,
for $0\le t \le 2\pi/\omega$.  The boundary of $\Sigma_1$ consists of two helices  $C_1$, $C_2$ together with two arcs $\alpha_1$, $\alpha_2$ both congruent to $\alpha$. 

We take the first variation of ${\mathcal E}_{a,\Lambda_0}$ with the variation field being the constant vector $E_3$. Since $E_3$ is the generator of a one parameter family of isometries,  this first variation vanishes. We express the first variation as in (\ref{fv}).  Since (\ref{con}) holds, the contributions to the boundary integral is zero since it is given by the right hand side of $(\ref{bint})$vanishes  and
the integrals over $\alpha_1$ and $\alpha_2$ cancel each other since these arcs are congruent and are traversed in opposite directions. We then obtain from the calculations given above, that
$$ 0=\int_{\Sigma_1} (-2H-\frac{a}{2}R^2+\Lambda_0)\nu_3\:d\Sigma\:,$$
which is a contradiction since the integrand is positive almost everywhere on $\Sigma_1$.

\end{proof}

This result can easily be modified for axially symmetric surfaces. In that case, the Killing field used is simply $E_3\times X$ and the helices are replaced by circles and the equation (\ref{con}) still holds.  
\section{TreamillSled coordinates analysis}

We will be considering immersions of the form

$$\phi(s,t)=(x(s) \cos(\omega t)+y(s)\sin(\omega t),-x(s)\sin(\omega t)+y(s)\cos(\omega t),t) $$

with the curve $\alpha(s)=(x(s),y(s))$  parametrized by arc-length. We will refer to the curve $\alpha$ as the profile curve of the surface since the surface given as the image of $\phi$ is the orbit of $\alpha$ under the helicoidal motion (\ref{HM}).  For $\theta(s)$ defined by
$$x^\prime(s)=\cos(\theta(s))\com{and} y^\prime(s)=\sin(\theta(s))$$
We define the {\it TreadmillSled coordinates}  $\xi_1(s)$ and $\xi_2(s)$ by
\begin{eqnarray}\label{TS coordinates}
 \xi_1(s)=x(s)\cos(\theta(s))+y(s)\sin(\theta(s)) \com{and}  \xi_2(s)=x(s)\sin(\theta(s))-y(s)\cos(\theta(s))\:.
\end{eqnarray}

The Gauss map of the immersion $\phi$ can be computed as
$$\nu=\frac{1}{\sqrt{1+w^2 \xi_1^2}}\left(\sin(\theta-\omega t),-\cos(\theta-\omega t),-\omega\xi_1\right)\:,$$
and so by a direct calculation, we obtain ${\hat Q}=\xi_2/\sqrt{1+\omega^2\xi_1^2}$. Finally, using that $R^2=x^2+y^2=\xi_1^2+\xi_2^2$, we see from (\ref{con}), that the immersion $\phi$ represents 
a rotating helicoidal drop if and only if there holds
\begin{equation}
G(\xi_1, \xi_2):= \frac{ 2 \xi_2}{\sqrt{1+\omega^2 \xi_1^2}} + \Lambda_0 (\xi_1^2+\xi_2^2) -\frac{a}{4} (\xi_1^2+\xi_2^2)^2\equiv {\rm constant}=:C\:.
 \end{equation}


After direct computation shows that  the equation $2 H= \Lambda_0-a \frac{R^2}{2}$ reduces to 

\begin{eqnarray}\label{mdtheta}
\theta^\prime(s)=\frac{ 2 w^2\xi_2  - 2 \Lambda_0 \left(1+w^2 \xi_1^2 \right)^{3/2} +a(\xi_1^2+\xi_2^2)\left(1+w^2 \xi_1^2\right)^{3/2}}{1+w^2
\left(\xi_1^2+ \xi_2^2 \right)} 
\end{eqnarray}

From the definition of $\xi_1$,  $\xi_2$ and $\theta$  we get that $\xi_1^\prime=1- \xi_2\theta^\prime$ and $\xi_2^\prime=\xi_1\theta^\prime$. Using Equation (\ref{mdtheta}) we conclude that $\xi_1$ and $\xi_2$ must satisfy

\begin{eqnarray} \label{the ode}
\xi_1^\prime &=&f_1(\xi_1,\xi_2) =\frac{  \left( 1+w^2 \xi_1^2 \right) \left(     2 +  \sqrt{ 1+w^2 \xi_1^2}\, \xi_2 \left( 2  \Lambda_0 -a(\xi_1^2+\xi_2^2)    \right)    \right)}{2(1+w^2 (\xi_1^2+ \xi_2^2 ))}   \\
\xi_2^\prime &=&f_2(\xi_1,\xi_2) = \frac{ \xi_1\,   \left(   2 w^2\xi_2 -  2  \Lambda_0  \left( 1+w^2 \xi_1^2 \right)^\frac{3}{2}  + a (\xi_1^2+\xi_2^2)
\left( 1+w^2 \xi_1^2 \right)^\frac{3}{2}        \right)}{2(1+w^2
(\xi_1^2+ \xi_2^2 ))}    
\end{eqnarray}

This system of ordinary differential equations for $\xi_1$ and $\xi_2$ provides a different proof of the fact that the $G(\xi_1(s),\xi_2(s))$ must be constant. Since we can check that

$$\frac{\partial G}{\partial \xi_1}= -\frac{2+2 w^2
(\xi_1^2+ \xi_2^2 )}{\left( 1+w^2 \xi_1^2 \right)^\frac{3}{2}}  \, f_2 \com{and} \frac{\partial G}{\partial \xi_2}= \frac{2+2 w^2
(\xi_1^2+ \xi_2^2 )}{\left( 1+w^2 \xi_1^2 \right)^\frac{3}{2}}  \,  f_1 $$


\begin{rem}\label{xi1positive} The level sets of $G$ are symmetric with respect to the $\xi_2$-axis, therefore in order to understand the level set of $G$, it is enough to understand those points in the level set with $\xi_1\ge 0$.

\end{rem}

In order to study the level sets of the function $G$ we change the variables $\xi_1$ and $\xi_2$ for the  variables $r$ and $\xi_2$ where

$$ r=\xi_1^2+\xi_2^2\:. $$

Doing this change, we get that the equation $G=C$ reduces to

$$ \frac{ 2 \xi_2}{\sqrt{1+\omega^2 r-\omega^2\xi_2^2}}   + \Lambda_0 r -\frac{a}{4} \, r^2= C \:.$$
In fact, this equation is exactly the one appearing in (\ref{con}).

Therefore,

$$ \xi_2=\frac{(4 C + r (-4 \Lambda_0 + a r))\sqrt{1+r \omega^2}}{\sqrt{64+(4C+r(-4\Lambda_0+ar))^2\, \omega^2}  } \:.$$

 By Remark \ref{xi1positive}, it is enough to consider those points with $\xi_1\ge0$. Since $\xi_1=\sqrt{r-\xi_2^2}$ we get that

 $$\xi_1=\frac{\sqrt{p(r,a,\Lambda,C)}}{\sqrt{64+(4C+r(-4\Lambda_0+ar))^2\, \omega^2}  } ,$$
 
where

\begin{eqnarray}
p(r,a,\Lambda,C)=-16 C^2 + 64 r + 32 C \Lambda_0 r - 16 \Lambda_0^2  r^2 - 8 a C r^2 + 8 a \Lambda_0 r^3 - a^2 r^4\, \:.
\end{eqnarray}

\begin{rem}
Since $p(r,a,\Lambda,C)$ is a polynomial in $r$ of degree 4 with negative leading coefficient when $a\ne0$ and $p$ is a polynomial of degree two when $a=0$, we get that the values of $r$ for which $p(r,a,\Lambda,C)$ is positive are bounded. Since $r=\xi_1^2+\xi_2^2=x^2+y^2$,  we conclude that the profile curve of any helicoidal rotating drop is bounded.
\end{rem}

\begin{mydef}\label{def of rho} Let $r_1$ and $r_2$ be two non negative values that satisfies $p(r_1)=p(r_2)=0$ and $p(r)>0$ for all $r\in (r_1,r_2)$. We define  $\rho:[r_1,r_2]\longrightarrow \bfR{2}$ by

$$ \rho(r)= \big(\,   \frac{\sqrt{p(r,a,\Lambda,C)}}{\sqrt{64+(4C+r(-4\Lambda_0+ar))^2\, \omega^2}  }  \, ,\,\frac{(4 C + r (-4 \Lambda_0 + a r))\sqrt{1+r \omega^2}}{\sqrt{64+(4C+r(-4\Lambda_0+ar))^2\, \omega^2}  }  \, \big) \, \:.$$

\end{mydef}

\begin{rem} \label{prop of rho} As pointed out before, all the level sets of the function $G$ are bounded. We have that  the map $\rho$ parametrizes half of the level set $G=C$.
\end{rem}

\begin{mydef}\label{fundamental piece}  In the case the level set $G=C$ is a regular closed curve or union of regular closed curves, we define a fundamental piece of the profile curve as a simple connected part of the profile curve such that the parametrized curve $(\xi_1,\xi_2)$ given by equations (\ref{TS coordinates}) correspond to exactly one closed curve in the  level set of $G=C$
\end{mydef}

\begin{rem}
From the definition of TreadmillSled given in \cite{P1}, we obtain that the profile curve of the solutions of the helicoidal rotational drop equation are characterized by the property that their treadmillSled are the level sets of $G$. In other words, using the notation of \cite{P1}, we have that $TS(\alpha)=\beta$ where $\beta$ is a parametrization of a connected component of the level set of $G=C$ and $\alpha$ is the profile curve of the helicoidal rotating drop. We will see that, for a few exceptional examples, the profile curve is a bounded complete curve having a circle as a limit cycle. For the non exceptional examples we can define an initial and final point of the fundamental piece and  we have that the whole profile curve is the union of rotated fundamental pieces. We also have that if  $R_1=\hbox{min}\{ |m| : m\in TS(\alpha)\}$ and $R_2=\hbox{max}\{ |m| : m\in TS(\alpha)\}$ and $\Delta{\tilde \theta}$ is the variation of the angle between $\overrightarrow{0p_1}$ and $\overrightarrow{0p_2}$ where $p_1$ and $p_2$ are the initial and final points of a fundamental piece, then, the profile curve is properly immersed if $\frac{\Delta\tilde{\theta}}{\pi}$ is a rational number, otherwise the profile curve is dense in the set
 $\{(x,y)\in \bfR{2}: R_1\le |(x,y)|\le R_2\}$.
 \label{def of delta theta}

\end{rem}

\begin{figure}[ht]
\centerline{\includegraphics[width=4cm,height=4cm]{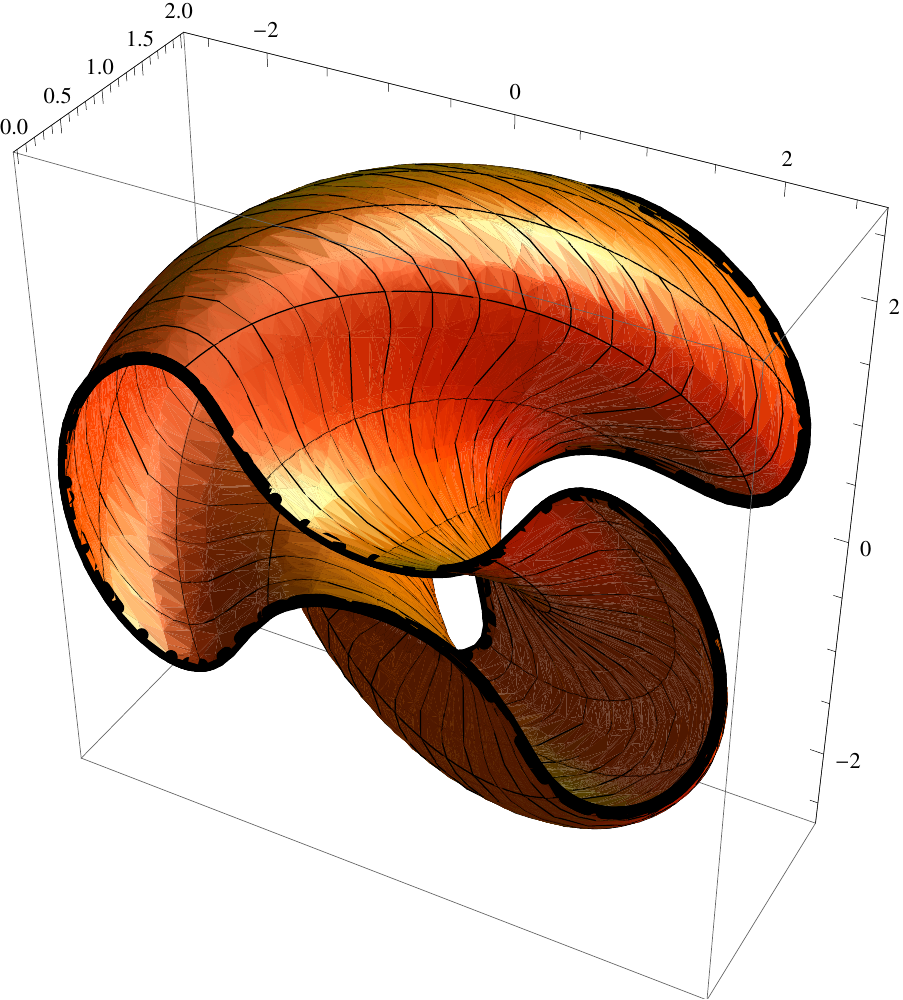}\hskip.1cm \includegraphics[width=4cm,height=3cm]{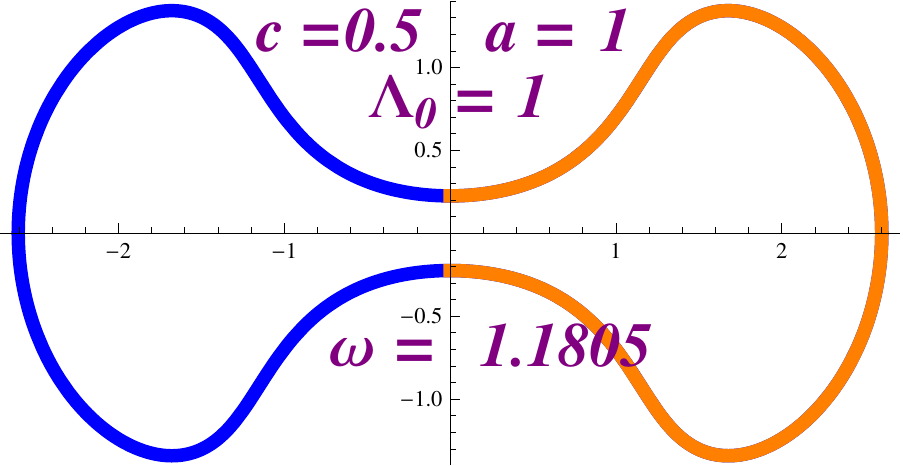}\hskip.1cm\includegraphics[width=3.5cm,height=3.5cm]{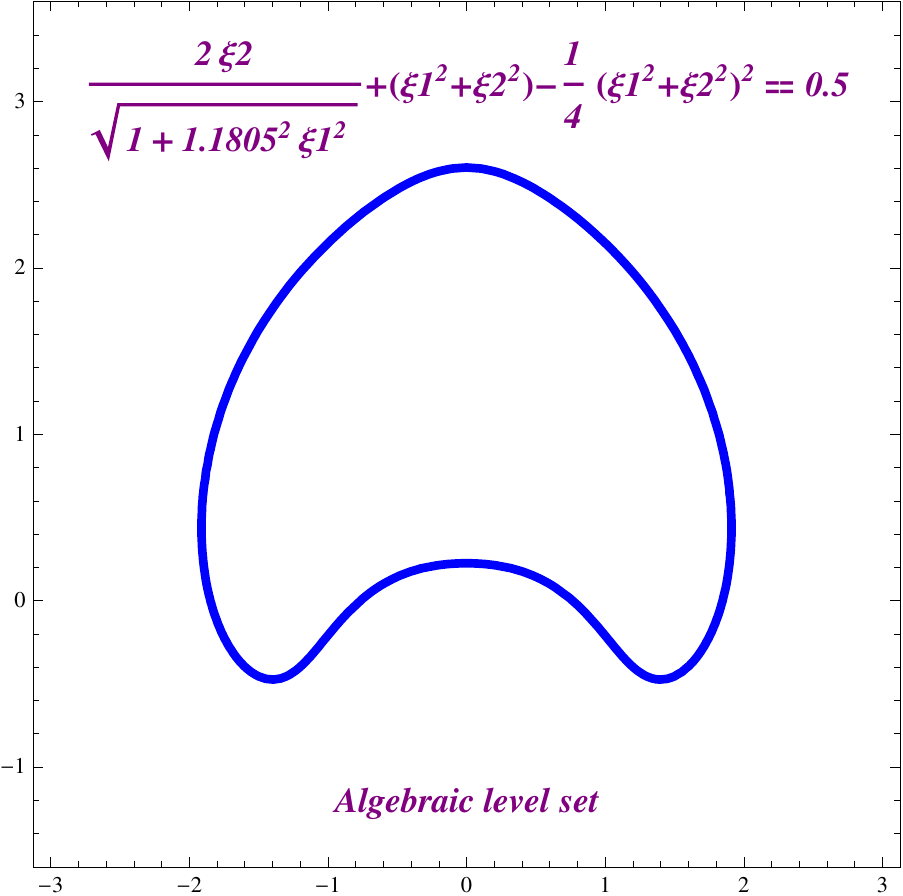}\hskip.1cm\includegraphics[width=3.5cm,height=3.5cm]{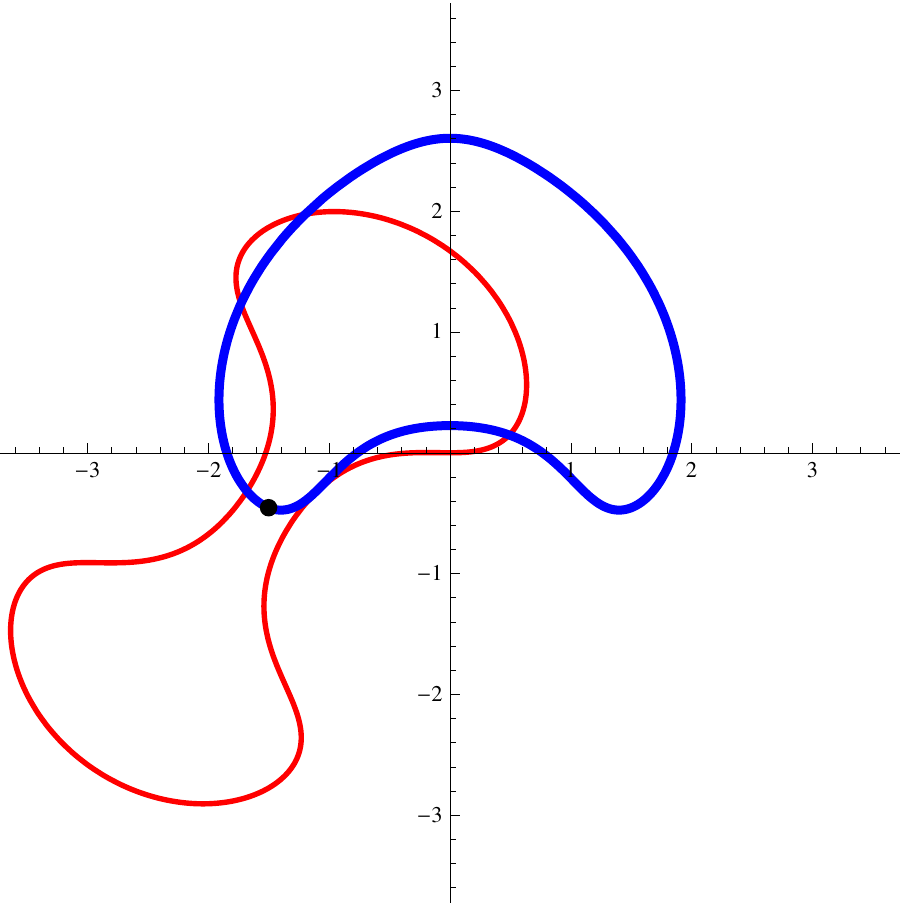}}
\caption{ The first picture shows a helicoidal rotational drop, the second picture shows its profile curve emphasizing the fundamental piece, the third picture show the level set $G=C$ and the last picture shows how the TreadmillSled of the profile curve produces the level set $G=C$, in this particular example the TreadmillSled of the profile curve will go over the level set $G=C$ two times.}
\label{TSofcyl}
\end{figure}

We compute the variation $\Delta{\tilde \theta}$ in terms of the parameter $r$. We assume that $\alpha(s)$ is the profile curve of a helicoidal rotational drop. Recall that we are assuming that that $s$ is the arc-length parameter for the curve $\alpha$. If $\beta(s)=(\xi_1(s),\xi_2(s))$, then we have that for some function $s=\sigma(r)$, $\beta(\sigma(r))=\rho(r)$ holds. By the chain rule we have that

\begin{eqnarray}\label{ds}
\frac{ds}{dr}=\frac{d\sigma}{dr}=\frac{|\rho^\prime(s)|}{|\beta^\prime(s)|}=\sqrt{\frac{|\rho^\prime(r)|^2}{f_1^2(s)+f_2^2(s)}}
=\frac{1}{2}\, \sqrt{\frac{64+(4C+r(-4\Lambda_0+ar))^2\, \omega^2}{p(r,a,C)}}\:.
\end{eqnarray}
and we also have that if $\tilde{\theta}$ denotes the polar angle of the profile curve, this is, if $\tilde{\theta}(s)$ satisfies the equation $ \alpha(s)=(x(s),y(s))=R(s) (\cos\tilde{\theta(s)},\sin\tilde{\theta(s)}) $, 
  then $\tilde{\theta}^\prime(s)= \frac{\xi_2(s)}{r}$


\begin{eqnarray}\label{dtheta}
\frac{d\tilde{\theta}}{dr} = \frac{d\tilde{\theta}}{ds}\, \frac{ds}{dr}= \frac{1}{2} \,  \frac{(4C+ a r^2-4 \Lambda_0 r) \sqrt{1+r\omega^2}}{r \sqrt{p(r,a,C)}}
\end{eqnarray}

Since the map $\rho:[r_1,r_2]\longrightarrow \bfR{{2}}$ parametrizes half of the TreadmillSled of the fundamental piece of the profile curve, we obtain the following expression for function $\Delta \tilde{ \theta}$ defined in Remark \ref{def of delta theta}

\begin{eqnarray}\label{change}
\Delta \tilde{\theta} = \Delta \tilde{\theta}(C,a,\omega,r_1,r_2)  =\, \int_{r_1}^{r_2}  \frac{(4C+ a r^2-4 \Lambda_0 r) \sqrt{1+r\omega^2}}{r \sqrt{p(r,a,C)}}\, dr
\end{eqnarray}

\begin{rem}
If we have a helicoidal rotational drop $\Sigma$ and we multiply every point by a positive fixed number $\lambda$, this is, if we consider the surface $\lambda \Sigma$, then this new surface satisfies the equation of the rotating drop for some other values of $\Lambda_0$ and $a$. We also have that if we change the orientation of the profile curve of a surface $\Sigma$ that satisfies the   equation of the rotating drop with values $\Lambda_0$, $a$ and $H$, then the reparametrized surface satisfies the equation with values  $-\Lambda_0$, $-a$ and $-H$.  With these two observations in mind, we have that in order to consider all the helicoidal rotational drops, up to parametrizations, rigid motions and dilations, it is enough to consider two cases: Case I, $\Lambda_0=0$ and $a=-1$ and Case II, $\Lambda_0=1$ and $a$ is any real number.

\end{rem}

\subsection{Case I: $\Lambda_0=0$ and $a=-1$}

In this case the polynomial $p(r,a,\Lambda,C)$ reduces to

$$q=q(r,C)=-16 C^2 + 64 r + 8 C r^2 - r^4 $$

Recall that we are interested in finding two positive consecutive roots of the polynomial $q$. Notice that when $C$ is a negative large number then the polynomial $q$ has no roots and when $C$ is a positive large number then the polynomial $q$ has more than one root. In every case $q(0)=-16 C^2\le0$ and the limit when $r\to \infty$  of $q(r)$ is negative infinity. The following lemma was proven in \cite{PP} and provides the number of possible roots of  $q(r,C)$ in terms of the values of $C$.

\begin{lem} \label{roots Lamda0=0}
For any $C>C_0=-\frac{3}{2^{\frac{2}{3}}}$, the polynomial $q(r,C)$ has exactly two positive real roots. 
When $C=C_0$, $\sqrt[3]{4}$ is the only real root of $q(r,C)$  and when $C<C_0$, $q(r,C)$ has not real roots.  
\end{lem}

Now we will compute the limit of $\Delta \tilde{\theta}$ when $C$ goes to $C_0$.
We will use the following lemma from \cite{P2}

\begin{lem}\label{lemma 1} Let $f(c,r)$ and $g(r,c)$ be smooth functions such that $g(C_0,r_0)= \frac{\partial g}{\partial r} (C_0,r_0)=0 $ and
$\frac{\partial^2 g}{\partial r^2} (C_0,r_0)= -2 A$ where $A>0$.  If $\{C_n\}$, $\{u_n\}$ and $\{v_n\}$ are sequences such that   $C_n$ converges to $C_0$, $u_n$ and $v_n$ converges to $r_0$ with $u_n<r_0<v_n$, and $g(u_n) =g(v_n)=0$  and  $g(r)>0$ for all $r\in(u_n,v_n)$, then

$$\int_{u_n}^{v_n}\frac{f(c,r)\, dr}{\sqrt{g(c,r)}} \longrightarrow  \, f(C_0,r_0)\, \frac{\pi}{\sqrt{A}} \com{as}  n\longrightarrow \infty\:. $$

\end{lem}

Notice that helicoidal rotating drops are defined when $C$ takes value from  $C_0=-\frac{3}{\sqrt[3]{4}}$ to $\infty$. When $C=C_0$ the only root of the polynomial $p$ is $r_0=\sqrt[3]{4}$. If we apply Lemma \ref{lemma 1} with $f(r,c)=\frac{(4 C-r^2) \sqrt{1+r\omega^2}}{r }$ and $g(r,c)=q(r,C)$ to the integral given in (\ref{change}), we obtained that

\begin{eqnarray}\label{lowerbound}
\lim_{C\to C_0^+} \Delta\tilde{\theta}= B(\omega) =- \frac{2 \pi}{\sqrt{3}} \sqrt{1+\sqrt[3]{4}\, \omega^2}
\end{eqnarray}

\begin{rem} \label{embedded}

Recall that whenever $\Delta \tilde{\theta}=\frac{n 2 \pi}{m}$ holds for some pair of  integers $m$ and $n$, then the entire  profile curve is properly immersed and it is invariant under the group $Z_m$. 

\end{rem}

\begin{rem}
Up to dilations and rigid motions, the moduli space for all helicoidal rotating drops with $\Lambda_0=0$ is the region in the plane

$$\{(C,\omega): C\ge C_0=-\frac{3}{\sqrt[3]{4}},\quad \omega>0 \}$$

Moreover,  for any $\omega>0$,  the surface associated with the point $(C,\omega)=(C_0,\omega)$ is a round cylinder of radius $\sqrt[3]{2}$, because it can be easily checked that when $C=C_0$, then, for any $\omega$, the level set $G=C$ reduces to the point $\{(0,-\sqrt{2})\}$

\end{rem}

\begin{figure}[h]\label{special examples}
\centerline{\includegraphics[width=8cm,height=6cm]{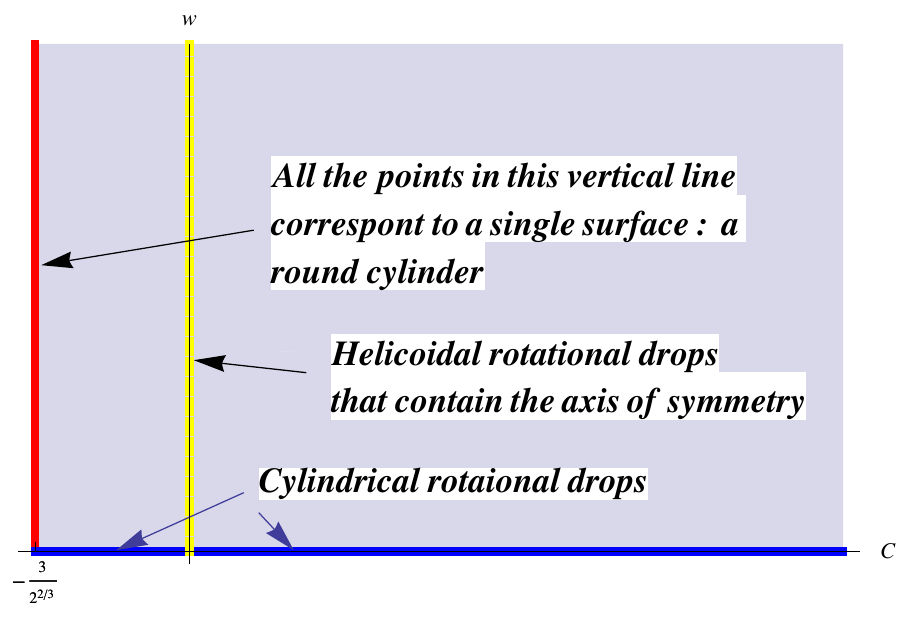}}
\caption{Moduli space of the helicoidal drops with $a=-1$ and $\Lambda=0$}
\end{figure}

\subsection{Case II: $\Lambda_0=1$} First  note that the case $a=0$ corresponds to helicoidal surface with constant mean curvature. These surfaces were studied using similar techniques in \cite{P1} and for this reason we will assume here that $a\ne0$. In this case the polynomial $p(r,a,\Lambda,C)$ reduces to 

\begin{eqnarray}\label{polyq}
q=q(r,C,a)=-16 C^2 + 64 r+32 C\,  r-16 \, r^2 - 8 a C\,  r^2+8 a \, r^3 -a^2\,  r^4\:.
\end{eqnarray}

Recall that we are interested in finding two positive consecutive roots of the polynomial $q$. The roots of the polynomial $q$ given in (\ref{polyq}) were analyzed in \cite{PP}. In order to describe the roots of $q$ we need to define the following functions.

\begin{mydef}\label{def of ri}
Let $h(R)=\frac{2(R-1)}{R^3}$,  and  define $R_1:(-\infty,0)\cup (0,\infty)\to {\bf R}$, $R_2:(0,8/27)\to {\bf R}$ and $R_3:(0,8/27)\to {\bf R}$ by

$$\begin{array}{cccc}
R_1(a)=R  \com{such that}&  h(R)=a & \com{with} &R<1\:,\\
R_2(a)=R \com{such that}& h(R)=a  &\com{with}  &1<R<3/2\:,\\
R_3(a)=R \com{such that} & h(R)=a  &\com{with} & 3/2<R<\infty \:.
\end{array}$$

We also define the functions $r_1:(-\infty,0)\cup (0,\infty)\to {\bf R}$, $r_2:(0,8/27)\to {\bf R}$ and $r_3:(0,8/27)\to {\bf R}$ by

$$ r_1(a)=R_1^2(a),\quad  r_2(a)=R_2^2(a),\quad  r_3(a)=R_3^2(a)\:. $$

\end{mydef}

 The following lemma was proven in \cite{PP} and provides the number of roots of $q$ depending on the values $a$ and $C$.

\begin{lem} \label{roots}
Let $r_1$, $r_2$ and $r_3$ be as in Definition \ref{def of ri} and for $i=1, 2, 3$  define
$$ C_i=\frac{16-8r_i+6a{r_i}^2-a^2{r_i}^3}{4(-2+ar_i)} $$

and

$$q=q(r, a,C)=-16 C^2 + 64 r+32 C\,  r-16 \, r^2 - 8 a C\,  r^2+8 a \, r^3 -a^2\,  r^4 \:.$$

Recall that the domain of $C_i$ is the same domain of $r_i$. That is, the domain of $C_1(a)$ is   $\{a\ne0\}$ and the domain of $C_2(a)$ and $C_3(a)$ is the interval  $(0,\frac{8}{27}]$.   For any $a\ne0$ and any $C$, the polynomial $q$ has non negative real roots whose multiplicities are given in the following table.

\begin{tabular}{| l | l | c |}
\hline
{\rm range of }$a$&{\rm range of} $C$&{\rm number of distinct real roots of }$q$\\ \hline

$a<0\Rightarrow C_1(a)<0$&$C<C_1(a)$&$0$\\
&$C=C_1(a)$&$1$\\
&$C>C_1(a)$&2\\
\hline \hline
$a\in (0,\frac{8}{27}) \Rightarrow C_2(a)<0$,  $C_1(a)>0$,\\[2mm] $C_2(a)<C_3(a)<C_1(a)$&$C>C_1(a)$&$0$\\
&$C=C_1(a)$&$1$\\
&$C_3(a)<C<C_1(a)$&$2$\\
&$C=C_3(a)$&$3$(****)\\
&$C_2(a)<C<C_3(a)$&$4$\\
&$C=C_2(a)$&$3$(***)\\
&$C<C_2(a)$&$3$\\
\hline \hline
$a=\frac{8}{27}\Rightarrow C_2(a)=C_3(a)=-9/8$, \\[2mm] $C_1(\frac{8}{27})=9$&$C<-9/8$&$2$\\
&$C=-9/8$&$2$(*)\\
&$-9/8<C<9$&$2$\\
&$C=9$&$1$(**)\\
&$9>C$&$0$\\
\hline \hline
$a> 8/27 \Rightarrow C_1(a)>0$& $C>C_1(a)$&$0$\\
&$C=C_1(a)$&$1$\\
&$C<C_1(a)$&$2$\\
\hline 
\end{tabular}

(*) In this case the roots are $9/4$ with multiplicity $3$ and $81/4$ with multiplicity one.\\
(**) In this case the only real root is $9$ with multiplicity $2$. \\
(***) In this case the first root has multiplicity 2.\\
(****) In this case the second root has multiplicity 2.
\end{lem}

Now that we  have discussed the roots of the polynomial $q$ we can describe the moduli space of all helicoidal rotating drops with $\Lambda=1$.

\begin{thm}\label{conf space b=1} Let $\Lambda_0=1$ and let $\Delta \tilde{\theta}$ be the function defined in (\ref{change}).  Let

$$\Omega_1=\{(a,C,\omega): C> C_1(a),\, a < 0, \, w>0\}\quad \Omega_2=\{(a,C,\omega): C_2(a) < C < C_3(a),\, 0<a< \frac{8}{27}\, , \, \omega>0\}\,$$

$$\Omega_3=\{(a,C,\omega): C< C_1(a),\, a > 0, \, \omega>0\}\quad \Omega =\Omega_1\cup\Omega_3\setminus \Omega_2$$

$$\beta_1=\{ (a,C,\omega)\, :\,  C=C_1(a),\,  a\ne 0,  \, \omega>0\}$$

$$\beta_2= \{ (a,C,\omega)\, :\,  C=C_2(a),\, 0< a<\frac{8}{27}, \, \omega>0  \} $$

$$\beta_3=\{ (a,C,\omega)\, :\,  C=C_3(a),\, 0< a< \frac{8}{27}, \, \omega>0 \} $$

Under the convention that a point $(a,C,\omega)$ represents a helicoidal rotating drop if the TreadmillSled of its profile curve is contained in the level set $G=C$, we have:

{\bf i.)}  Every point $(a,C,\omega)$ in the interior of  $\Omega$ represents a helicoidal rotating drop with its fundamental piece having finite length. The TreadmillSled of the profile curve of these surfaces are parametrized by $\rho$ defined for values of $r$ between the only two roots of the polynomial $q(r,a,C)$.

{\bf ii.)}  Every point in  $\Omega_2 $ represents two helicoidal rotating drops, both having fundamental pieces of finite length. The TreadmillSleds of the profile curves of these surfaces are parametrized by $\rho$ defined for those values of $r$  that lie between the first and second root of the polynomial $q(r,a,C)$ and the third and fourth root of the polynomial $q(r,a,C)$ respectively.

{\bf iii.) } Every point $(a,C,\omega)$ in the set $\beta_1$ represents a circular helicoidal rotating drop. This cylinder is the same for all values of $\omega$.

{\bf iv.)} Every point $(a,C,\omega)$ in the set $\beta_2$ represents two helicoidal rotating drops:  a circular cylinder and a non circular cylinder with bounded length of its fundamental piece. The circular cylinder is the same for all values of $w$.

{\bf v.)} Every point in the set $\beta_3$ represents three helicoidal rotating drops. One is a circular cylinder, which is the same for all values of $w$. The second one has a TreadmillSled parametrized by $\rho$ defined for those values of $r$  that lie between the first and second root of the polynomial $q(r,a,C)$. Recall that  the second root has multiplicity $2$. The third surface   has a TreadmillSled parametrized by $\rho$ defined for those values of $r$  that lie between the second and third root of the polynomial $q(r,a,C)$. The second and third surfaces are not properly immersed and their profile curves have a circle as a limit cycle and they have infinite winding number with respect to a point interior to this circle. Solutions similar to these second or third types  will be called  \textbf{helicoidal drops of exceptional type}.

{\bf vi.)} Points of the form $(a,C,\omega)=(\frac{8}{27},-\frac{9}{8} ,\omega)$ represent two helicoidal rotating drops:  a circular cylinder, which is the same for all values of $\omega$, and one helicoidal drop of exceptional type.

{\bf vii.)} Up to a rigid motion, every helicoidal drop falls into one of the cases above.

{\bf viii.)} Every helicoidal drop that is not exceptional is either properly immersed (when $\frac{\Delta\tilde{\theta}(a,C,\omega)}{\pi}$ is a rational number) or it is dense in the region bound by two round cylinders (when $\frac{\Delta\tilde{\theta}(a,C,\omega)}{\pi}$ is an irrational number).

\end{thm}

\begin{figure}[ht]
\centerline{\includegraphics[width=9cm,height=5.5cm]{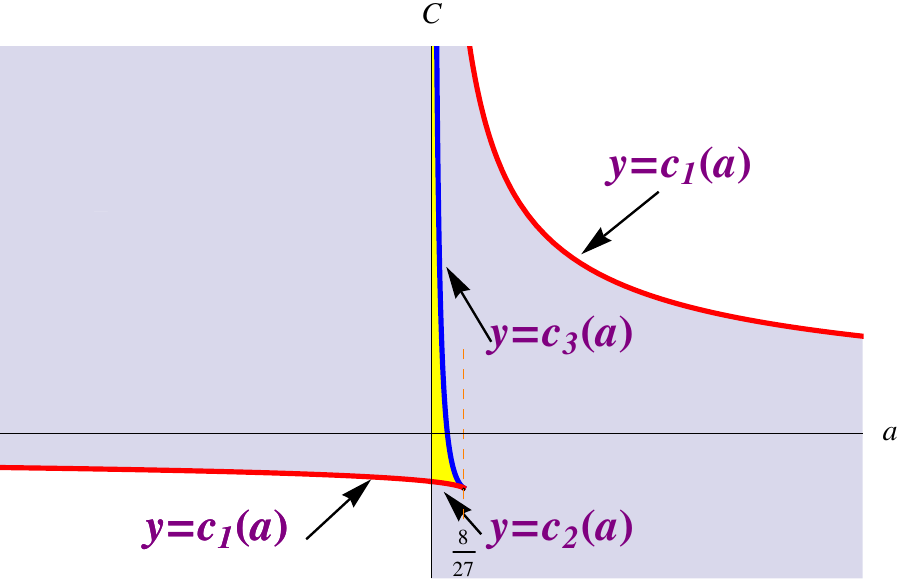}}
\label{c graphs}
\caption{ The picture above shows the graph of the functions $C_1$, $C_2$ and $C_3$, these graphs are used in Proposition \ref{conf space b=1} to describe the moduli space of all helicoidal rotational drops with $\Lambda_0=1$   }

\end{figure}

\begin{proof}

We already know that the TreadmillSled of the profile curve of any helicoidal rotating drop satisfies the equation

 $$G(\xi_1,\xi_2)=  \frac{ 2 \xi_2}{\sqrt{1+\omega^2 \xi_1^2}} + \Lambda_0 (\xi_1^2+\xi_2^2) -\frac{a}{4} (\xi_1^2+\xi_2^2)^2 = C$$

We also know that, up to rigid motions, the TreadmillSled of a curve determines the curve, see \cite{P1}.  Since any level set of $G$  can be parametrized using the map $\rho$ given in Definition \ref{def of rho}, and every parametrization  of a level set of $G$ is defined for values of $r$ where the polynomial $q$ is positive,  it then follows from Lemma \ref{roots} that every helicoidal rotating drop can be represented as one of the cases i, ii, iii, iv, v and vi. It is worth recalling, see Remark \ref{prop of rho}, that the parametrization $\rho$ only covers half of the level set of the map $G$. Each one of these level sets is  symmetric with respect to the $\xi_2$ axis, and the parametrization $\rho$ covers the half on the right.

Notice that when the profile curve is a circle, the level set $G=C$ reduces to a point. When the profile curve is a circle we will take the parametrization $\rho$ to be defined just in a point, a root with multiplicity 2 of the polynomial $q$.

When case (i) occurs, $q$ has  only  two simple roots $x_1$ and $x_2$ with $x_1<x_2$. We can check that the derivative of $q$ at $x_1$ is positive while the derivative of $q$ at $x_2$ is negative, so the length of the fundamental piece,  according to Equation (\ref{ds}), reduces to  $\int_{x_1}^{x_2}\, \sqrt{\frac{64+(4C+r(-4\Lambda_0+ar))^2\,  \omega^2}{q(r,a,C)}}\, dr $ which converges. Therefore the length of the fundamental piece is finite.


For values of $C$, $\omega$  and $a$ that fall into  case (ii),  the polynomial q has 4 roots $x_1<x_2<x_3<x_4$ and it is positive from $x_1$ to $x_2$ and from $x_3$ to $x_4$. Also, the level set of $G$ has two connected components. Half of each connected components of $G=C$ can be parametrized using the map $\rho$. One half of the connected component of $G=C$ uses the domain $(x_1,x_2)$ for $\rho$ and the half of the other connected component of $G=C$ uses the domain $(x_3,x_4)$ for $\rho$. The proof that the length of the fundamental piece of each surface is finite follows as in the proof in case (i).


For values of $(a,C,\omega)$ that satisfies the case (iii),  the polynomial $q$ has only one root $x_1=r_1$ with multiplicity two. We  take $R=\sqrt{x_1}$.  A direct calculation shows that if $a>0$, then $R_1(a)=-R$ and if we consider the profile curve $\alpha(s)=(R \sin(\frac{s}{R}),-R\cos(\frac{s}{R}))$, then $\xi_1=0$, $\xi_2=R$ and $G(\xi_1,\xi_2)= 2R+\Lambda_0R^2-\frac{a}{4}R^4$. Using the definition of $C_1$  and the fact that $R_1(a)=-R$, we can check that the expression $G(\xi_1, \xi_2)$ reduces to $C=C_1(a)$, which was our goal in order to show that the point $(a,C,\omega)$ represents a round cylinder.  Similarly, a direct verification shows that  if $a<0$, then $R_1(a)=R$ and if we consider the profile curve $\alpha(s)=(R \sin(\frac{s}{R}),R\cos(\frac{s}{R}))$, then $\xi_1=0$, $\xi_2=-R$ and $G(\xi_1,\xi_2)= -2R+\Lambda_0R^2-\frac{a}{4}R^4$. Using the definition of $C_1$  and the fact that $R_1(a)=R$, we can check that the expression $G(\xi_1, \xi_2)$ reduces to $C=C_1(a)$.  Since $r_1$ is independent of $w$ we have that these cylinders are independent of the value of $w$. This finish the proof of part (iii).


For values of $(a,C,\omega)$ that fall into case (iv), the polynomial $q$ has three roots $x_1<x_2<x_3$, where $x_1=r_1$ has multiplicity two and $x_2$ and $x_3$ are simple. The polynomial $q$ is positive for values of $r$ between $x_2$ and $x_3$. In this case the level set $G=C$ is the union of the point $(0,-\sqrt{x_1})$ and a closed curve.  If  we consider the cylinder or radius
 $\sqrt{x_1}$ oriented by the inward pointing normal, then, a direct computation shows that its mean curvature is $2 H=1-\frac{a}{2} r_1$.  Therefore this circular cylinder is a helicoidal rotating drop for the given parameters. Note that this cylinder is independent of $w$. The TreamillSled of  the profile curve of the other rotating drop is the level closed curve component of $G=C$; half of this part can be parametrized by the map $\rho$ with domain those values of $r$ between $x_2$ and $x_3$.



For values of $(a,C,\omega)$ in case (v), the polynomial $q$ has only three roots $x_1<x_2<x_3$, where $x_2=r_2$ has multiplicity two and $x_1$ and $x_3$ are simple. The polynomial $q$ is positive for values of $r$ between $x_1$ and $x_2$ and for values of $r$ between $x_2$ and $x_3$.  In this case the level set $G=C$ is connected but it self-intersects at the point $(0,-\sqrt{x_2})$. Any part of a curve that crosses the $\xi_2$-axis non-horizontally cannot be the TreadmillSled of a regular curve (see Proposition 2.11 in \cite{P3}). Therefore the correct way to view the level set $G=C$ in this case is not as a connected closed curve that self intersects but as the union of two curves and a point. Figure \ref{levelsetG} shows one of these level sets.


\begin{figure}[ht]
\centerline{\includegraphics[width=5cm,height=5cm]{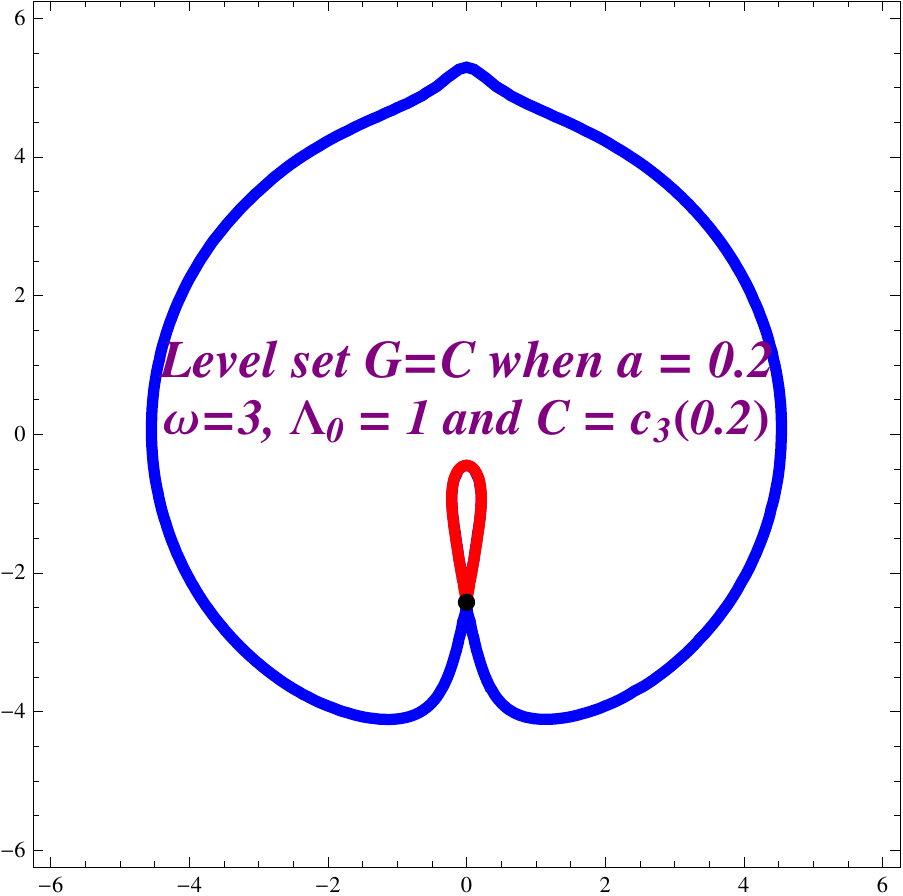}}
\caption{  For helicoidal drops in case (v) of Proposition \ref{conf space b=1}, the level set of $G$ should be regarded as the union of two curves and a point. Each curve is the TreadmillSled of an exceptional helicoidal rotating drop and the point is the treadmillSled of a circular cylinder .}.
\label{levelsetG}
\end{figure}

 One can check that the circular cylinder with radius $\sqrt{r_2}$, oriented by the inward pointing normal satisfies the $2 H=1-\frac{a}{2} r_2$, therefore this circular cylinder is a helicoidal rotating drop for the given parameters. The TreadmillSled associated with the profile curve of this round cylinder reduces to the point $(0,-\sqrt{x_2})$. The set $G=C$ minus the point $(0,-\sqrt{x_2})$ has two connected components. One of these connected components can be parametrized using the map $\rho$ with values of $r$ between $x_1$ and $x_2$ and the other using the map $\rho$ with values of $r$ between $x_2$ and $x_3$. Each of these connected components  is the TreamillSled of the fundamental curve for a rotating helicoidal drop whose length is unbounded. Specifically, their lengths are given respectively by the divergent integrals $\int_{x_1}^{x_2} \sqrt{\frac{64+(4C+r(-4\Lambda_0+ar))^2\, \omega^2}{q(r,a,C)}}\, dr$ and $\int_{x_2}^{x_3} \sqrt{\frac{64+(4C+r(-4\Lambda_0+ar))^2\, \omega^2}{q(r,a,C)}} \, dr$.  Moreover, using the definition of TreadmillSled, we notice that the function giving the distance to the origin of the profile curve, $|(x(s),y(s)|$, agrees with the function distance to the  origin of the level set $G=C$ given by $|(\xi_1(s),\xi_2(s))|=|\rho(\sigma^{-1}(s))|$. Therefore  as $r$ aproaches $r_2$, $s=\sigma(r)$ goes to $-\infty$ and the function $|(x(s),y(x)|$ approaches $\sqrt{r_2}$. Since polar angle of the profile curve can be calculated by integrating the expression in (\ref{dtheta}), we conclude that $\tilde{\theta}(r)$ also goes to $-\infty$ as $r$ approaches $r_2$. We conclude that the profile curve has a circle of radius $\sqrt{r_2}$ as a limit cycle and it has infinite winding number with respect to a point interior to this circle.



For values of $(a,C,w)$ that satisfies the case (vi) the polynomial $q$ has only two roots $x_1<x_2$ where $x_1=\frac{9}{4}$ has multiplicity three and $x_2=\frac{81}{4}$ is simple. The polynomial $q$ is positive for values of $r$ between $x_1$ and $x_2$. In this case the level set $G=C$ is connected but it has a singularity  at the $(0,-\frac{3}{2})$.  We can check that the circular cylinder with radius $\frac{3}{2}$ oriented by the inward pointing normal satisfies the $2 H=1-\frac{a}{2} r_1$, therefore this circular cylinder is a helicoidal rotating drop. The TreadmillSled associated of the profile curve of this round cylinder reduces to the point $(0,-\frac{3}{2})$. The set $G=C$ minus the point $(0,-\frac{3}{2})$ is connected and half of it  can be parametrized using the map $\rho$ with values of $r$ between $\frac{9}{4}$ and $\frac{81}{4}$. This part of the set $G=C$ is the TreamillSled of the fundamental curve of a rotating helicoidal drop whose length is not bounded.



Since we know that the profile curve of every rotating helicoidal drop satisfies the integral equation $G=C$ and cases (i)-(vi) cover all the possible level sets for the level sets of $G$ then every rotating helicoidal drop fall into one of the first 6 cases of this proposition. This proves (vii). 

In order to prove (viii) we notice that when a helicoidal rotating drop is not exceptional, it has a fundamental piece with finite length whose TreadmillSled is a closed regular curve (a connected component of the set $G=C$). By the properties of the TreadmillSled operator (in particular the one that states that the TreadmillSled inverse is unique up to rotations about the origin), we have the the whole profile curve is a union of rotations of the fundamental piece. The angle of rotation is given by $\Delta\tilde{\theta}=\Delta\tilde{\theta}(C,a,\omega,x_1,x_2)$. We therefore have that the profile curve is invariant under the group of rotations

\begin{eqnarray}\label{the group}
\mathbb{G}=\{(y_1,y_2)\rightarrow (\cos(n \Delta\tilde{\theta}) y_1 +  \sin(n \Delta\tilde{\theta}) y_2 , -\sin(n \Delta\tilde{\theta}) y_1+  \cos(n \Delta\theta) y_2 )\, \vert \, n\in \mathbb{Z}\} \:.
\end{eqnarray}

It is clear that if $\frac{\Delta\tilde{\theta}}{\pi}$ is a rational number then the group $\mathbb{G}$ is finite and the helicoidal surface is properly immerse. Moreover, if $\frac{\Delta\tilde{\theta}}{\pi}$ is not a rational number, then the group $\mathbb{G}$ is not finite and the helicoidal drop is dense in the region bounded by the two cylinders of radius $\sqrt{r_1}$ and $\sqrt{r_2}$. A more detailed explanation of this last statement can be found in \cite{P1}.


\end{proof}

\subsection{Embedded and properly embedded examples}

In this subsection we will find some embedded examples and we will show their profile curves. As pointed out before, when the helicoidal drop is not exceptional, its profile curve is 
a union rotations of fundamental pieces that ends up being invariant under the group $\mathbb{G}$ define in (\ref{the group}). It is not difficult to see that a necessary condition for the helicoidal drop to be embedded is that $\Delta \tilde{\theta}=\frac{2 \pi}{m}$ for some integer $m$. We will show that this condition is not sufficient. In order to catch the potentially embedded examples we need to understand the function $\Delta \tilde{\theta}(C,a,\omega,x_1,x_2)$. As a very elementary technique to solve the equation $\Delta \tilde{\theta}=\frac{2 \pi}{m}$ we will use the intermediate value theorem. We know that for any $a$, there is a first (or last) value of $C$, $C_0(a,x_1,x_2)$, for which the function $\Delta \tilde{\theta}$ is defined. We will compute the limit of $\Delta \tilde{\theta}$ when $C$ goes to $C_0$ using Lemma \ref{lemma 1}. The graphs shown in this paper were generated using the software Mathematica 8 .

\subsection{Embedded examples with $\Lambda_0=0$ and $a=-1$} From Lemma \ref{roots Lamda0=0} we know that the polynomial $p$ has two positive root if and only if $C>C_0=-\frac{3}{\sqrt[3]{4}}$. A direct application of Lemma \ref{lemma 1} shows:

\begin{prop}
If $\Lambda_0=0$,   $a=-1$ and for any $C>C_0$, $x_1$ and $x_2$ denote the two roots of the polynomial $q(r,C)=-16 C^2 + 64 r + 8 C r^2 - r^4 $, then 

$$\lim_{C\to C_0^+} \Delta\theta(C,\omega,x_1,x_2)=\, \int_{x_1}^{x_2}  \frac{(4C- r^2) \sqrt{1+r\omega^2}}{r \sqrt{q(r,C)}}\, dr=B(\omega)=
-\frac{2 \pi \sqrt{1+\sqrt[3]{4}\, \omega^2} }{\sqrt{3}}$$
holds.
\end{prop}


\begin{figure}[ht]
\centerline{ \includegraphics[width=4cm,height=3cm]{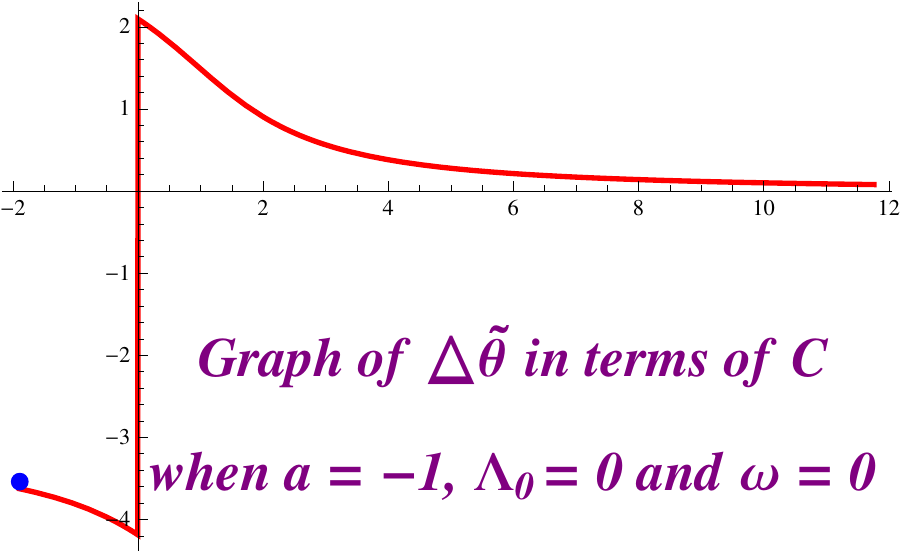}\includegraphics[width=4cm,height=3cm]{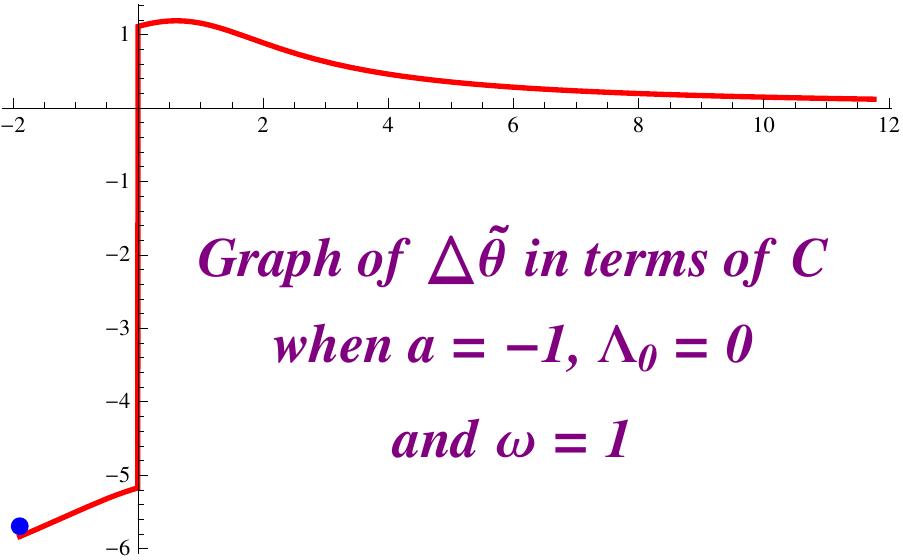}\includegraphics[width=4cm,height=4cm]{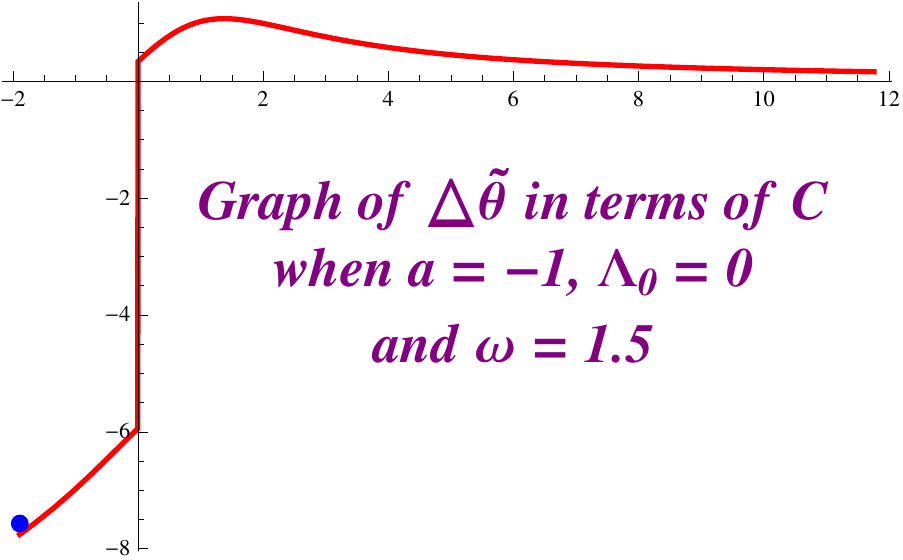}\includegraphics[width=3.5cm,height=3cm]{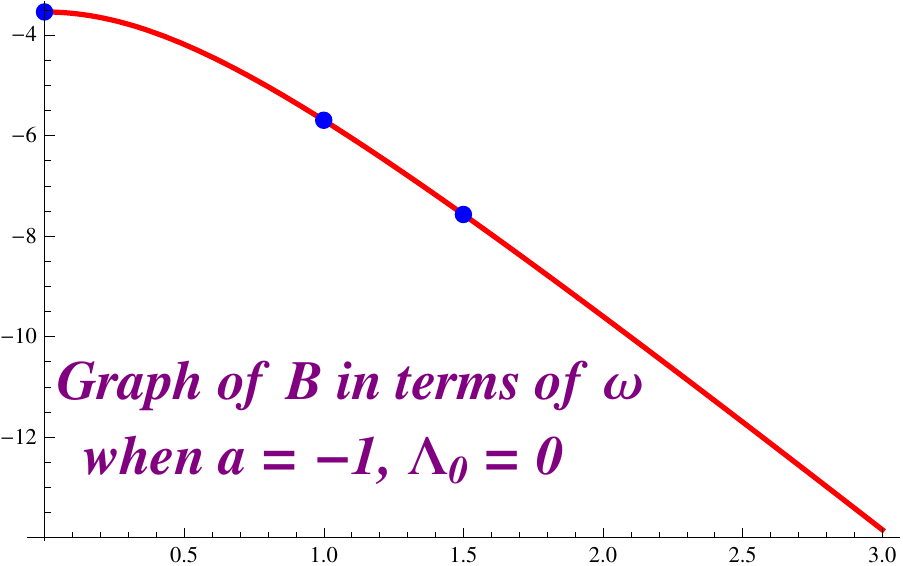}}
\label{limit1}
\caption{ The last graph shows how the beginning of the graphs on the left change when $\omega$ changes. The highlighted points in the last graph ($\omega=0$, $1$ and $1.5$) correspond to the highlighted points in the three graphs to the left. For $\omega=0$ there is no solution of the equation $\Delta{\tilde \theta}=-\frac{2\pi}{m}$ with negative values of $C$. We see that,  for some values of $\omega$, the equation $ \Delta{\tilde \theta}=-2\pi$ has a solution with $C$ negative, which is responsible for the existence of embedded examples with $\Lambda_0=0$ and $a=-1$. }
\end{figure}


\begin{figure}[ht]
\centerline{\includegraphics[width=3.5cm,height=3.5cm]{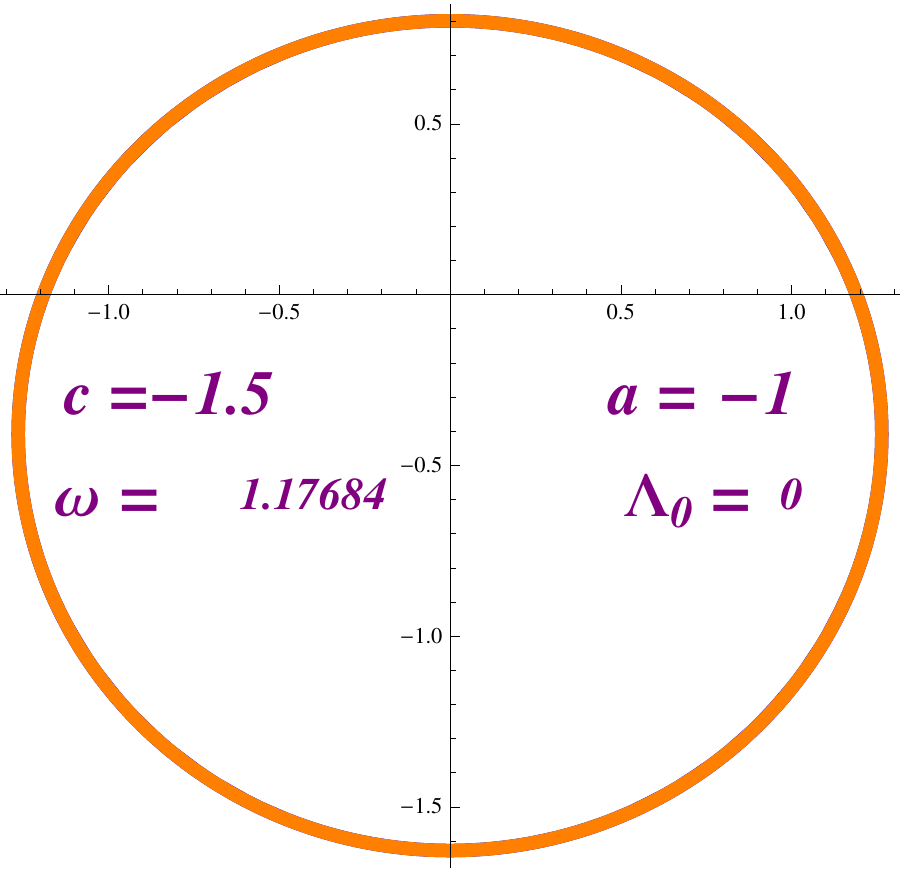}\includegraphics[width=3.5cm,height=3.5cm]{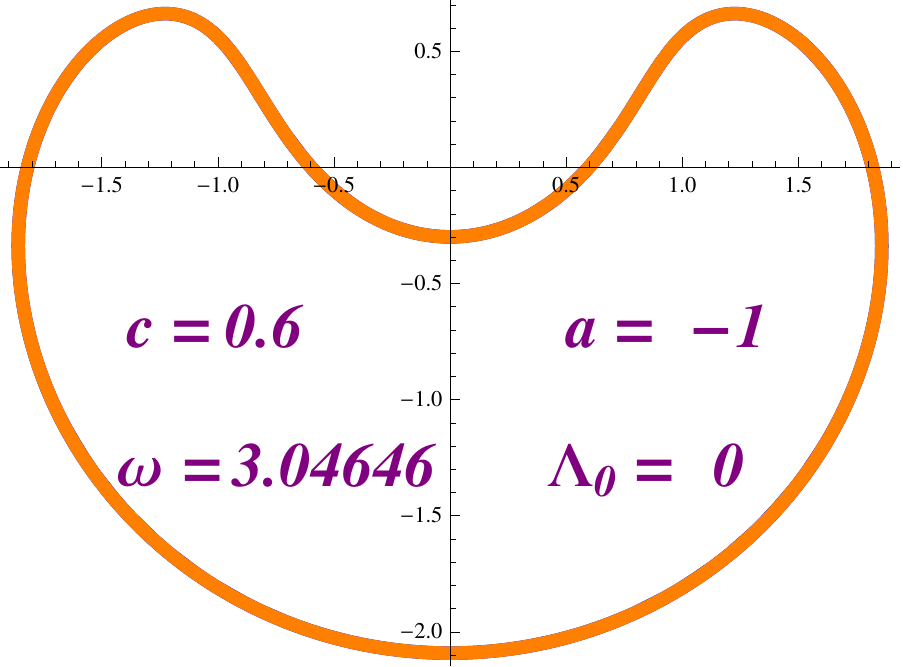}\includegraphics[width=3.5cm,height=3.5cm]{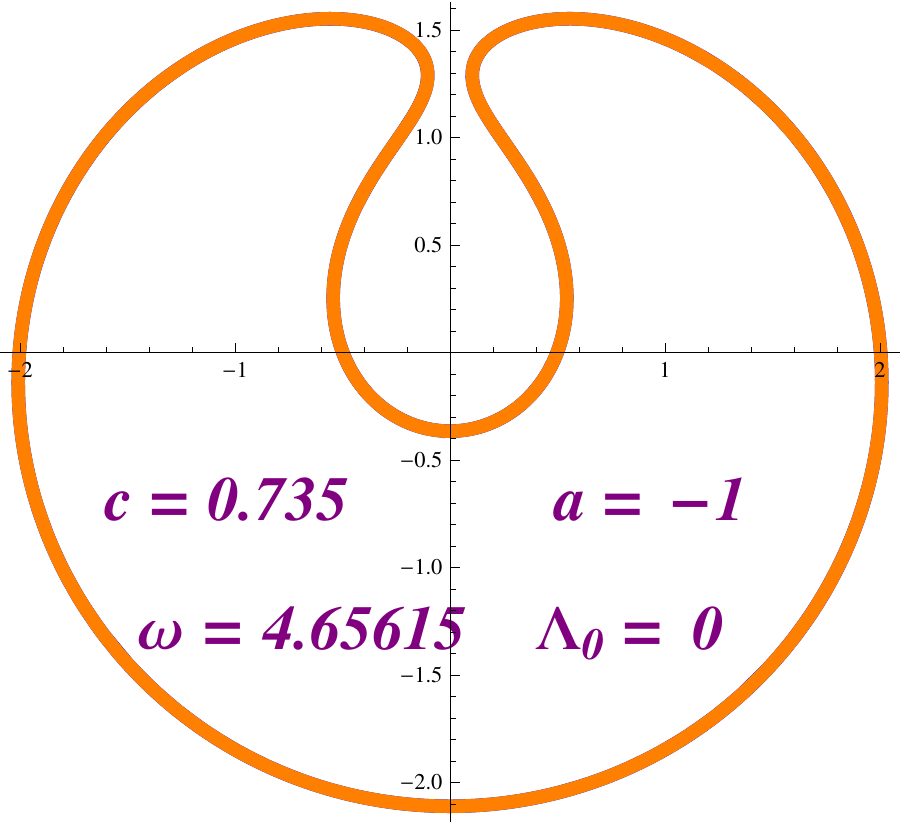}}
\caption{ The profile curves of some embedded helicoidal rotating drops when $\Lambda_0=0$ and $a=-1$. When $C$ is close to the critical value $C_0,$ the embedded examples is close to a round cylinder. As $C$ increases, the shape develops a self-intersection.}
\label{embset1}
\end{figure}

\begin{figure}[ht]
\centerline{\includegraphics[width=3.5cm,height=3.5cm]{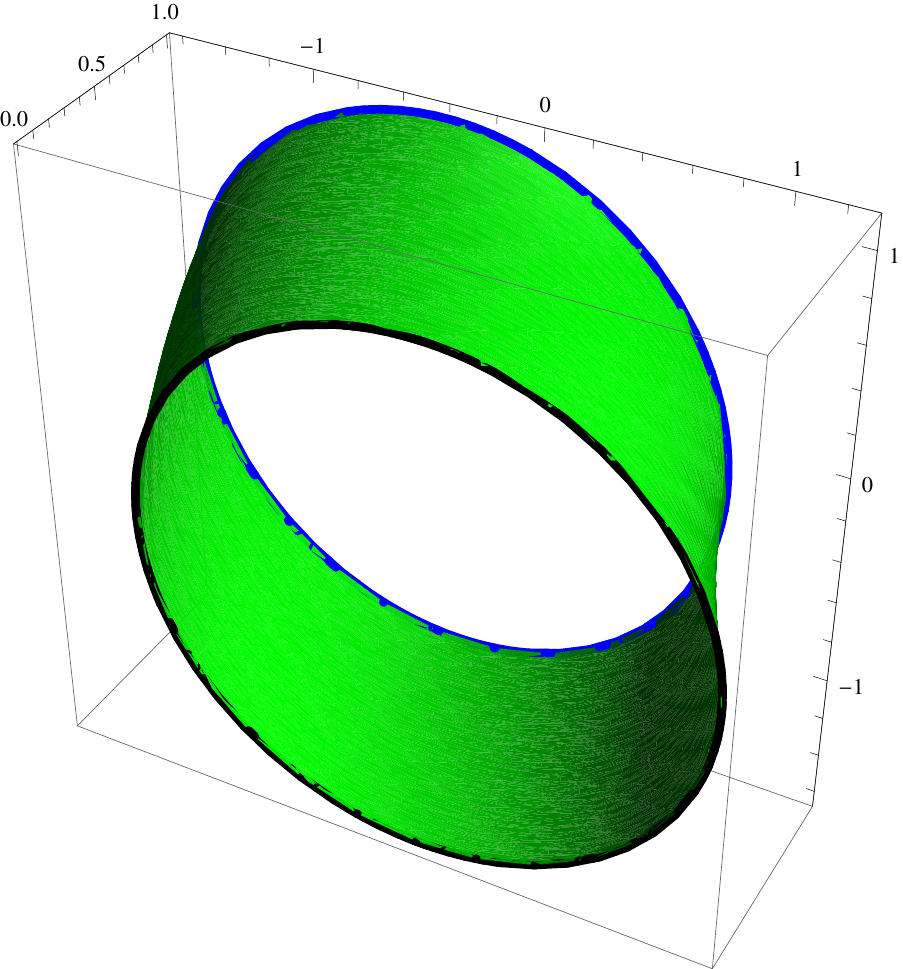}\includegraphics[width=3.5cm,height=3.5cm]{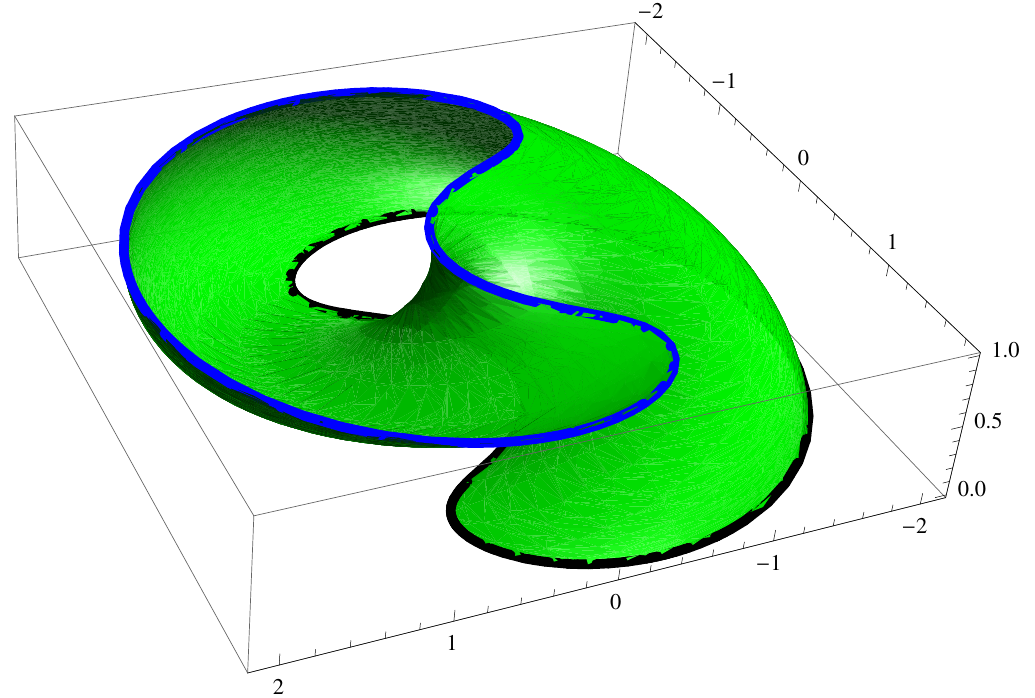}\includegraphics[width=3.5cm,height=3.5cm]{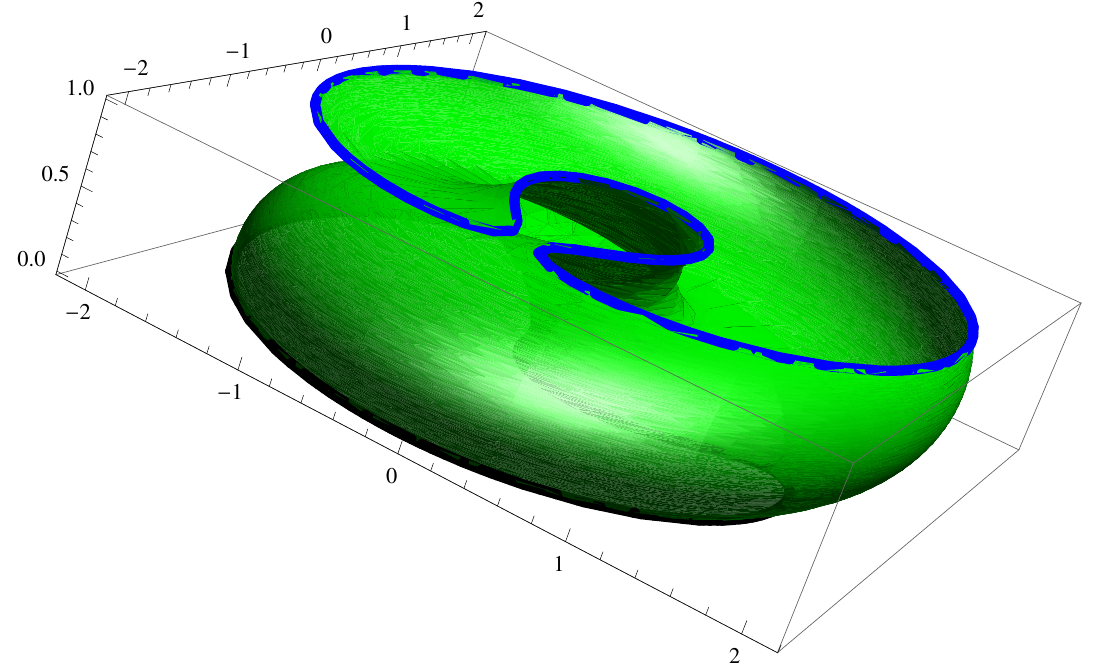}}
\caption{ Surfaces associated with the profile curve on Figure \ref{embset1}\:.}
\label{embset22}
\end{figure}

Using the Intermediate Value Theorem,  we can numerically solve the equation  $ \Delta\tilde{ \theta}=-2\pi$  for values of $w$ and $C$. The images in Figures \ref{embset1} and \ref{embset22} show some of the resulting profile curves and the corresponding surfaces. They also show the values of $C$ and $w$ that solve the equation $ \Delta \tilde{ \theta}=-2\pi$.

\subsection{Embedded examples with $\Lambda_0=1$ and $a\ne 0$}

We now show  some embedded examples in this case. Again  the Intermediate Value Theorem is used to numerically solve the equation 
$\Delta\tilde{\theta}=\frac{2 \pi}{m}$. A direct application of Lemma \ref{lemma 1} shows:

\begin{prop}
Let $r_1$ $r_2$ and $r_3$ be the expressions given in Definition \ref{def of ri} and let $C_1$, $C_2$ and $C_3$ be the expression defined in Lemma \ref{roots}. For $i=1,2 $ let us defined the following two bounds.

$$b_i(a,w)= \frac{\pi (4 C_i + a r_i^2-4r_i) \sqrt{1+ r_i w^2}}{r_i \sqrt{16 + 8 a (C - 3 r_i) + 6 a^2 r_i^2}}\:.$$

{\bf a.)} If $a<0$ then $\lim_{C\to C_1(a)^+}\Delta\tilde{ \theta}(C,a,x_1,x_2)=b_1(a)$. Here $x_1$ and $x_2$ are the first two roots of the polynomial $q(r,C,a)$

{\bf b.)} If $a>0$ then $\lim_{C\to C_1(a)^-}\Delta\tilde{ \theta}(C,a,x_1,x_2)=b_1(a)$. Here $x_1$ and $x_2$ are the first two roots of the polynomial $q(r,C,a)$

{\bf c.)} If $a>0$ then $\lim_{C\to C_2(a)^+}\Delta\tilde{ \theta}(C,a,x_1,x_2)=b_2(a)$. Here $x_1$ and $x_2$ are the first two roots of the polynomial $q(r,C,a)$.

\end{prop}

\begin{proof}
 
Since  $\Delta\tilde{ \theta}=\int_{x_1}^{x_2}\frac{(4C+a r^2-4r)\sqrt{1+r\omega^2}}{r\sqrt{q}} \, dr$,  in every case, when $C$ approaches the limit value, the two roots approach $r_i$ (i=1 or 2),  which is a root of $q$ with multiplicity $2$. Therefore Lemma \ref{lemma 1} applies and the proposition follows. Notice that the value $A$ in Lemma  \ref{lemma 1} is given by 

$$A=-\frac{1}{2}q^{\prime\prime}(r_i)=16 + 8 a (C - 3 r_i) + 6 a^2 r_i^2\:.$$

\end{proof}

\begin{rem} When $0<a<\frac{8}{27}$ the domain of the function $\Delta\tilde{\theta}$ starts at $C=C_2(a)$ and the limit of the function $\Delta\tilde{\theta}(C)$ when $C\to C_2(a)^+$ is $b_2(a,w)$. Moreover, this function has a vertical asymptote at $C=C_3(a)$, and then it has a jump discontinuity at $C=0$. Finally, this function is define for all  values of $C$ smaller than $C_1(a)$ and the limit of $\Delta\tilde{\theta}(C)$ when $C\to C_1(a)^-$ is $b_1(a,w)$

\end{rem}

\begin{figure}[ht]
\centerline{\includegraphics[width=6.5cm,height=4cm]{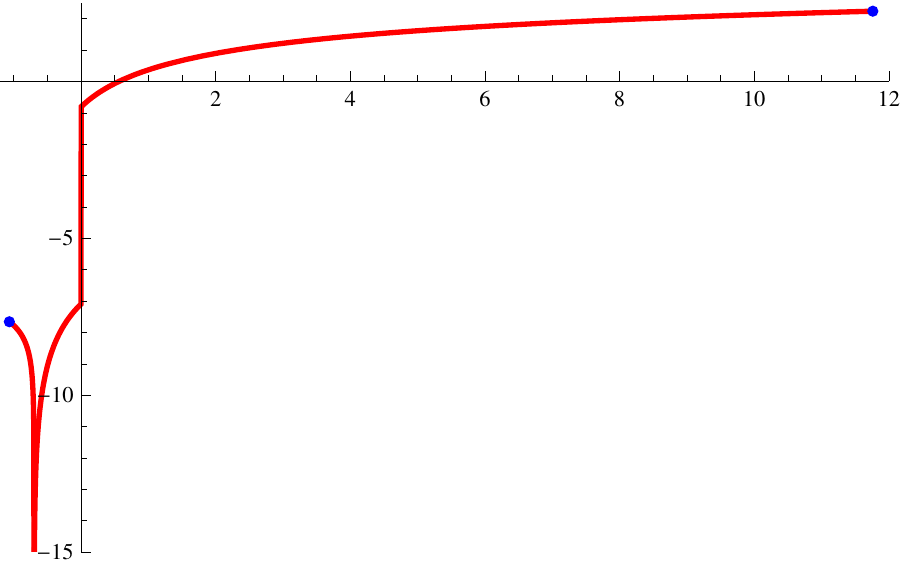}}
\caption{ The graph of the function $\Delta\tilde{\theta}$ when $a=0.2$ and $\omega=0.15$. In this case $C_2\approx -1.065$, $b_2\approx -7.66$, $C_3\approx -0.698$, $C_1\approx 11.76$ and $b_1\approx 2.23$. The points $(C_2,b_2)$ and $(C_1,b_1)$ have been highlighted.}
\label{graphD}
\end{figure}

Taking a look at Figure \ref{graphD} we notice that, when $ \omega=0.15$ and $a=0.2$, we have that for any integer  $m>2$, the equation $\Delta\tilde{\theta}=\frac{2 \pi}{m}$ has a solution. We have numerically solved this equation for $m=4$ and $m=8$. Figures \ref{hemb1} and \ref{hemb2} provides a picture of the profile curves of the  properly immerse examples.

\begin{figure}[ht]
\centerline{\includegraphics[width=5.5cm,height=3.5cm]{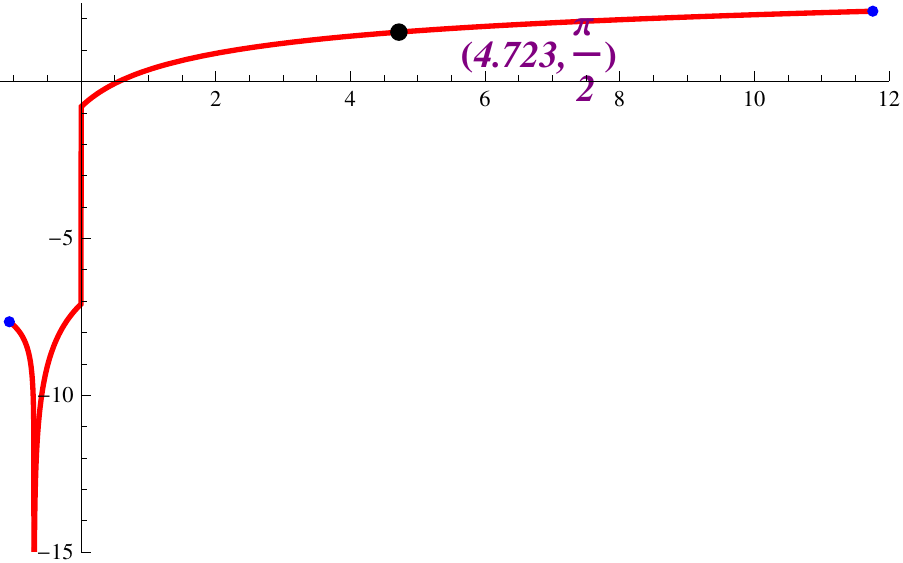}\hskip.3cm \includegraphics[width=3.5cm,height=3.5cm]{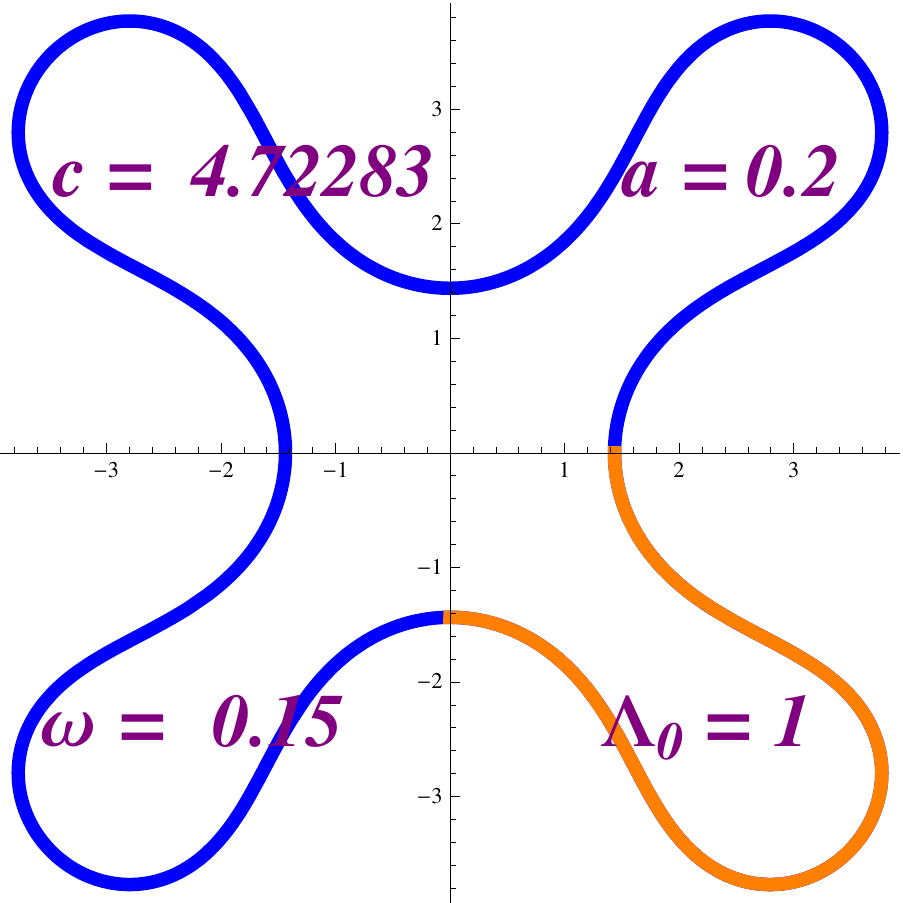}\hskip.3cm\includegraphics[width=3.5cm,height=3.5cm]{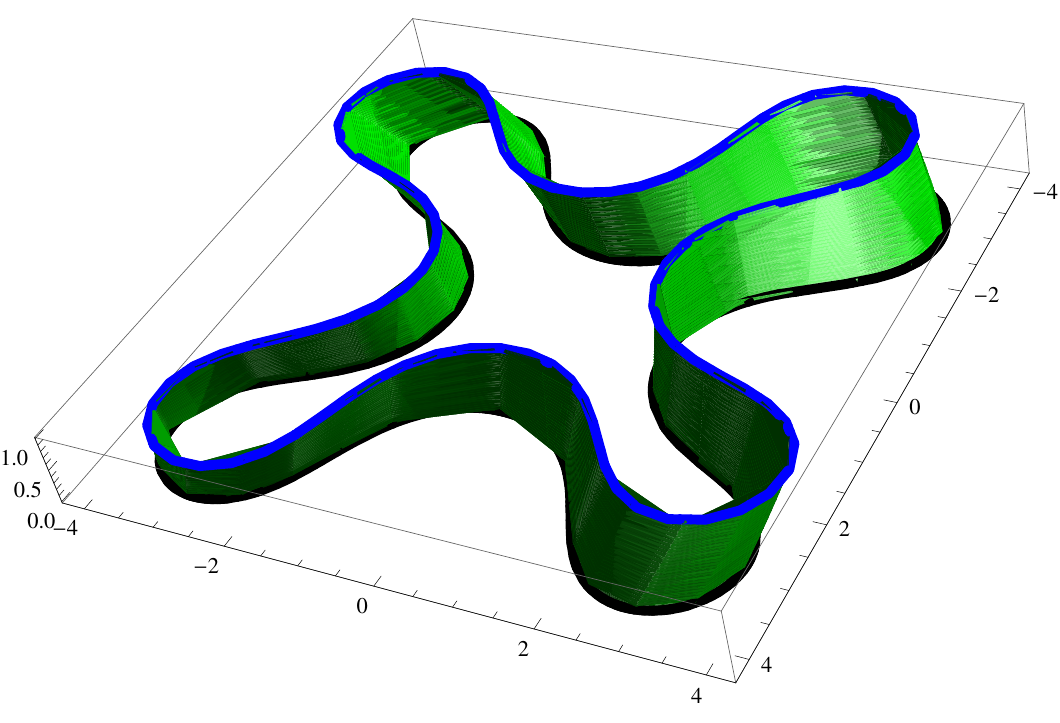}}
\caption{ The embedded helicoidal rotational drop obtained by solving the equation
 $\Delta\tilde{\theta}=\frac{\pi}{2}$ when $a=0.2$ and $\omega=0.15$.  }
\label{hemb1}
\end{figure}

\begin{figure}[ht]
\centerline{\includegraphics[width=5.5cm,height=3.5cm]{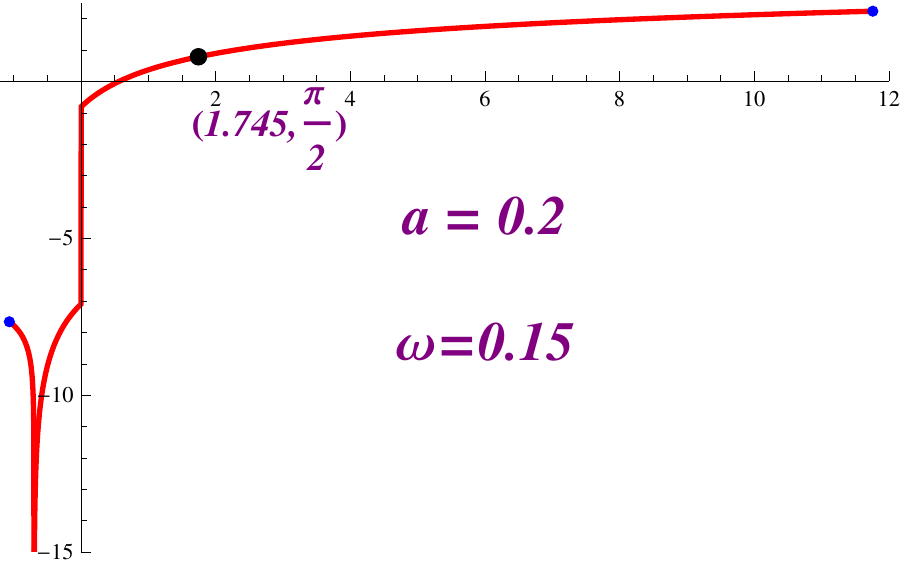}\hskip.3cm\includegraphics[width=3.5cm,height=3.5cm]{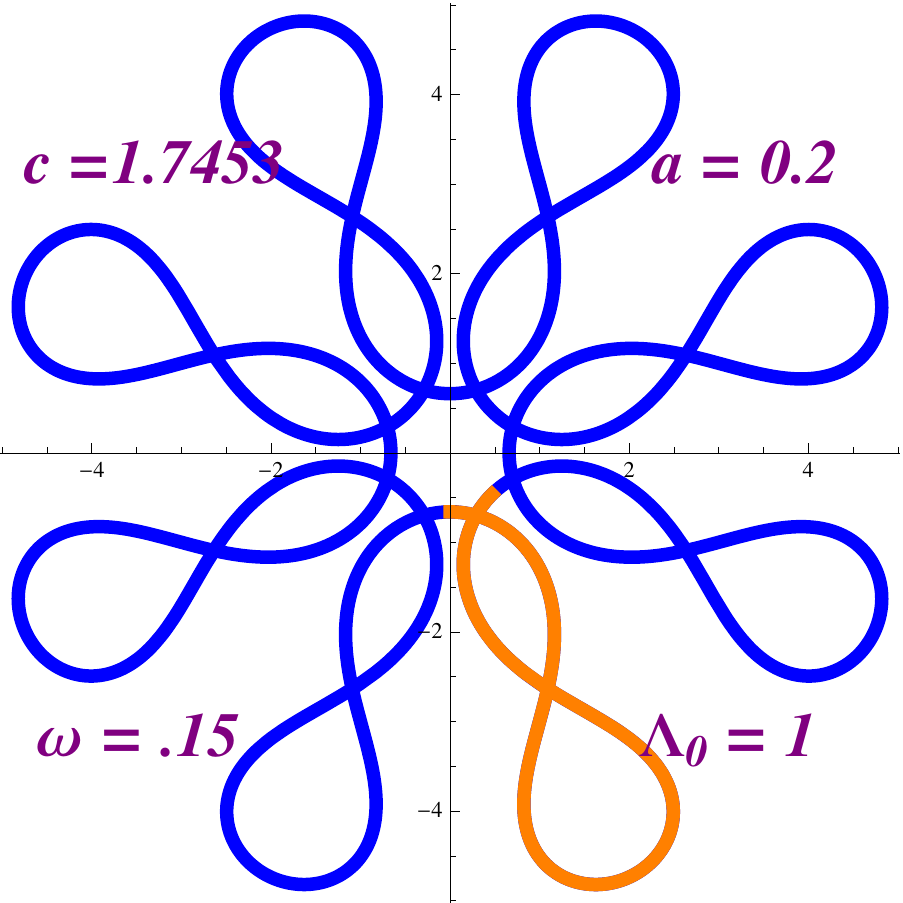}}
\caption{ The properly immerse helicoidal rotational drop obtained by solving the equation
 $\Delta\tilde{\theta}=\frac{\pi}{4}$ when $a=0.2$ and $\omega=0.15$.   }
\label{hemb2}
\end{figure}

We have that if we decrease the value of $a$ while keeping the value of $\omega$  constant, we can again solve the equation  $\Delta\tilde{\theta}=\frac{\pi}{4}$ but this time the helicoidal rotational drop is embedded. See Figure \ref{hemb3}.

\begin{figure}[ht]
\centerline{\includegraphics[width=5.5cm,height=3.5cm]{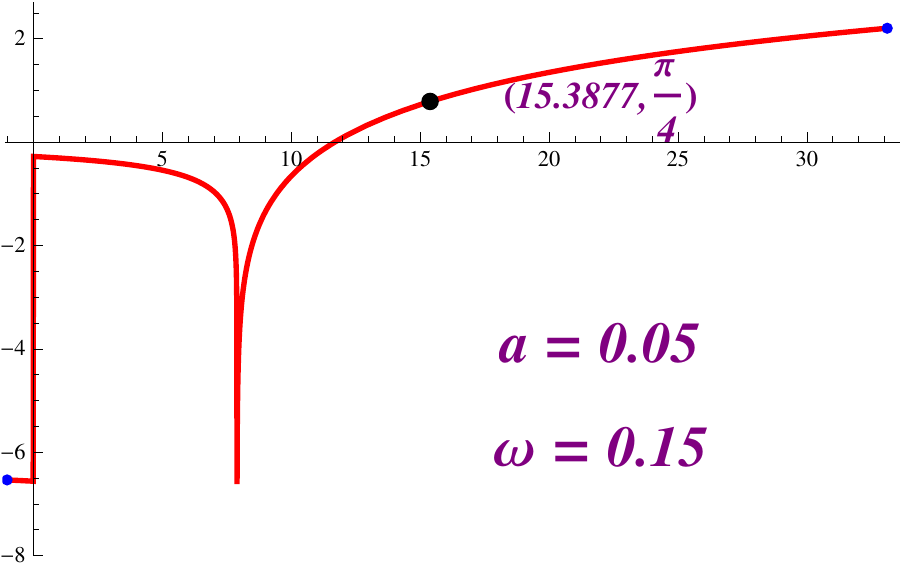}\hskip.3cm\includegraphics[width=3.5cm,height=3.5cm]{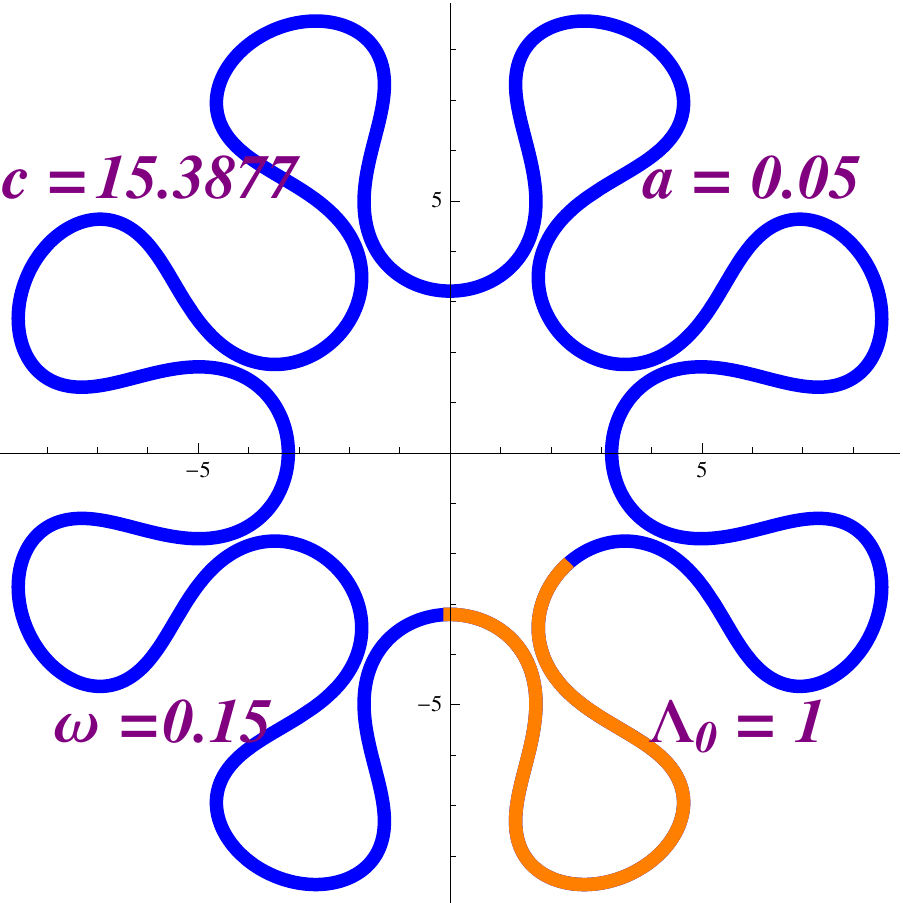}}
\caption{ The embedded helicoidal rotational drop obtained by solving the equation
 $\Delta\tilde{\theta}=\frac{\pi}{4}$ when $a=0.05$ and $\omega=0.15$.  }
\label{hemb3}
\end{figure}

Finally we would like to show that if we increase  $\omega$ then it is possible to solve the equation $\Delta\tilde{\theta}=2 \pi$, (see Figure    \ref{hemb4}) .

\begin{figure}[h!]
\centerline{\includegraphics[width=5.5cm,height=3.5cm]{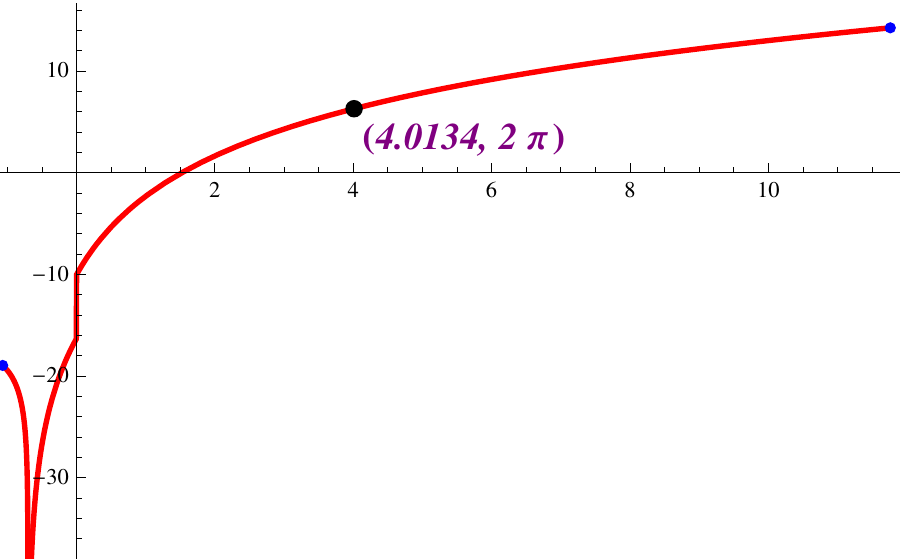}\hskip.3cm \includegraphics[width=3.5cm,height=3.5cm]{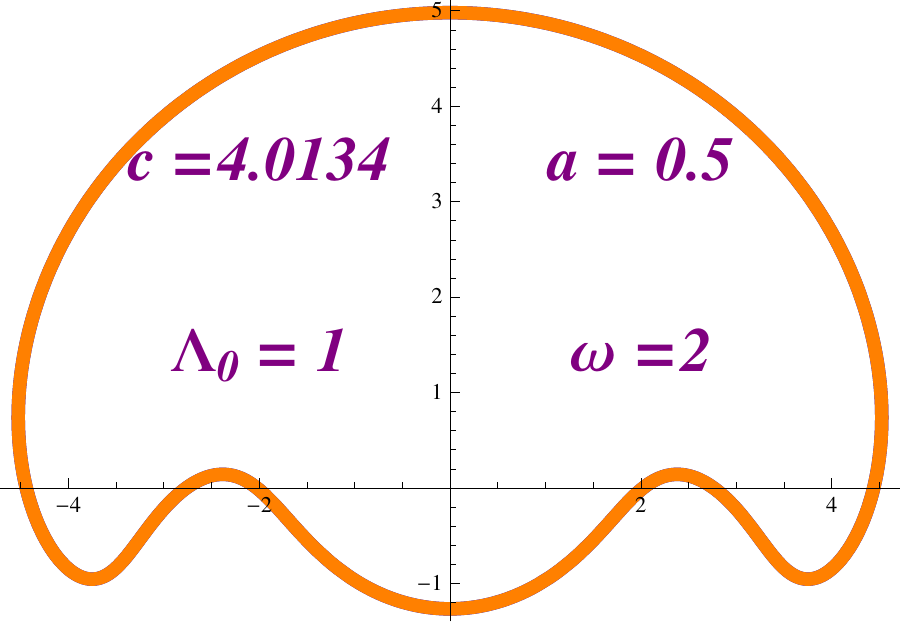}\hskip.3cm\includegraphics[width=3.5cm,height=3.5cm]{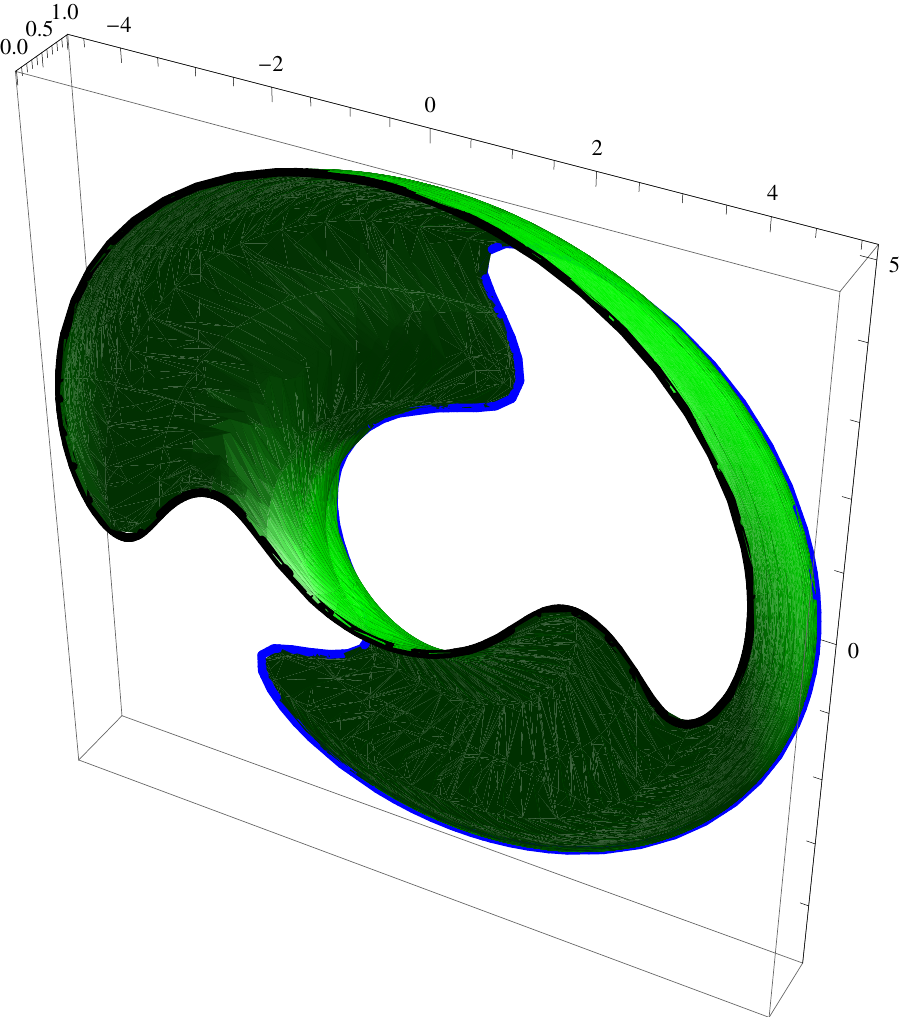}}
\caption{ The embedded helicoidal rotational drop obtained by solving the equation
 $\Delta\tilde{\theta}=2 \pi$ when $a=0.2$ and $\omega=2$. }
\label{hemb4}
\end{figure}



\vfil
\FloatBarrier

\section{Second variation}

For any sufficiently smooth surface, we define an invariant $\ell=2H+(a/2)R^2$. The first variation formula (\ref{fv})  restricted to  compactly supported variations
can then be expressed
$$\delta {\mathcal E}_{a,\Lambda_0}=-\int_\Sigma (\ell-\Lambda_0)\psi \:d\Sigma\:,$$
where $\psi:=\delta X\cdot \nu$. We assume that the surface is in equilibrium so that $\ell-\Lambda_0\equiv 0$ holds. The second variation is thus
$$\delta^2 {\mathcal E}_{a,\Lambda_0}=-\int_\Sigma \psi (\delta \ell)\:d\Sigma\:.$$
A well known formula for the pointwise variation of the mean curvature is
\begin{equation}
\label{dH}
2\delta H={\hat L}[\psi]+2\nabla H\cdot \delta X\:,\end{equation}
where ${\hat L}=\Delta +|d\nu|^2$.  Also
\begin{eqnarray}
\delta R^2=2\sum_{i=1,2} x_i\delta X\cdot E_i&=&2\sum_{i=1,2} x_i(\psi \nu_i+(\delta X)^T\cdot E_i)\nonumber \\
&=&2\psi {\hat Q}+2\nabla 'R^2\cdot (\delta X)^T\:\nonumber,
\end{eqnarray}
where ${\hat Q}=x_1\nu_1+x_2\nu_2$.  Combining this with (\ref{dH}), we have
\begin{equation}
\label{dl}
\delta \ell=L[\psi]+\nabla \ell\cdot T\:,
\end{equation}
where $L[\psi]=\Delta \psi +(|d\nu|^2+a{\hat Q})\psi\:$. Since we are assuming $\ell\equiv \Lambda_0=$ constant, the second term above vanishes and the second variation formula for variations vanishing on $\partial \Sigma$ then reads
\begin{equation}
\label{sv}
\delta {\mathcal E}_{a, \Lambda_0}=-\int_\Sigma  \psi L[\psi]\:d\Sigma =-\int_\Sigma \psi (\Delta \psi +(|d\nu|^2+a{\hat Q}\psi)\:d\Sigma=\int_\Sigma |\nabla \psi|^2-(|d\nu|^2+a{\hat Q})\psi^2\:d\Sigma\:.
\end{equation}
This formula can be found in \cite{RL}.
As usual, an equilibrium  surface  will be called  stable if the second variation is non negative for all compactly supported variations satisfying the additional condition
\begin{equation}
\label{zint}\int_\Sigma \psi\:d\Sigma =0\:.\end{equation}
This is just the first order condition which is necessary and sufficient for the variation to be volume preserving.

For a part of the surface of the form $\alpha \times [-h/2, h/2]$ this can be written
\begin{eqnarray}
\label{AA}
\delta^2 {\mathcal E}_{a, \Lambda_0} &=& 
 \int_\alpha  \int _{-h/2}^{h/2} \biggl( \frac{1}{\sqrt{1+\omega^2\xi_1^2}}( [1+\omega^2 R^2]\psi_s^2 -2\omega \xi_2\psi_s\psi_t+\psi_t^2)\\ \nonumber
&& - (4H^2-2K+a\xi_2)\sqrt{1+\omega^2\xi_1^2}\psi^2\biggr)\:dt\:ds\:,
 \end{eqnarray}
 where $K$ denotes the Gaussian curvature.
Choosing  $\psi=\sin((2\pi t)/h)$, gives
$$ \int _\Sigma |\nabla \psi|^2\:d\Sigma= \frac{2\pi^2}{h}\int_\alpha  \frac{1}{\sqrt{1+\omega^2\xi_1^2}}\:ds\:.$$
In addition, for this choice of $\psi$, we have $\psi\equiv 0$ on the boundary and the mean value of $\psi$ on $\alpha \times [-h/2, h/2]$ is zero.

\begin{lem}
There holds
$$\int_\alpha K\sqrt{1+\omega^2\xi_1^2}\:ds=0\:,$$
and hence
$$\int _{\alpha \times [-h/2, h/2]}K\:d\Sigma =0\:.$$
\end{lem}
\begin{proof} From calculations found in \cite{P1}, one finds

$$K= \frac{-\omega^2(1+\kappa \xi_2)}{(1+\omega^2\xi_1^2)^{2}}=\frac{-\omega^2(\xi_1)_s} {(1+\omega^2\xi_1^2)^{2}}$$
$$\int_\alpha K\sqrt{1+\omega^2\xi_1^2}\:ds=\int_\alpha \frac{-\omega^2(\xi_1)_s} {(1+\omega^2\xi_1^2)^{3/2}}\:ds=0\:,$$
since the last integrand is the $s$ derivative  of a function of $\xi_1$. \end{proof}

\begin{prop}
A necessary condition for the stability of  $\alpha \times [-h/2, h/2]$ for the fixed boundary problem is that
\begin{equation}
\label{bound1}  \frac{2\pi^2}{h^2}\int_\alpha  \frac{1}{\sqrt{1+\omega^2\xi_1^2}}\:ds\ge \int_\alpha 4H^2 \sqrt{1+\omega^2\xi_1^2}\:ds +a{\mathcal A}\:\end{equation}
holds. Equivalently, this can be expressed as
\begin{equation}
\label{bound2} \frac{2\pi^2}{h^2}\int_\alpha  \frac{1}{\sqrt{1+\omega^2\xi_1^2}}\:ds\ge \int_{ \alpha \times [-h/2, h/2]} 4H^2 \:d\Sigma +a{\mathcal V}(\alpha \times [-h/2, h/2])\:.\end{equation}
\end{prop}
\begin{proof}
We choose   $\psi=\sin((2\pi t)/h)$ in the second variation formula. For this choice of $\psi$, we have $\psi\equiv 0$ on the boundary and the mean value of $\psi$ on $\alpha \times [-h/2, h/2]$ is zero.
The result then follows directly from (\ref{AA}) and the previous lemma.
\end{proof}
The bound (\ref{bound1}) gives a condition on the maximum height of a stable helicoidal surface in terms of the geometry of the generating curve.

There is no possible way to obtain a positive lower bound for the right hand side of (\ref{bound1}). For a round cylinder of radius $R$, the equation
$\frac{2\xi_2}{\sqrt{1+\omega^2\xi_1^2}}+\Lambda_0R^2-\frac{aR^4}{4}=c$ becomes $2R+\Lambda_0R^2-\frac{aR^4}{4}=c$, so for arbitrary $a$, we can simply define $c$ by this equation so and hence the cylinder will be an equilibrium surface. For a cylinder, the potential in the second variation formula is
$$4H^2-2K+a{\hat Q}=\frac{1}{R^2}+aR\:,$$
so for $a<<0$ the potential is non positive and the cylinder is  stable for arbitrary heights.

We will now give an upper bound for the height of a stable helicoidal equilibrium surface which is valid for any such surface which is not a cylinder over a planar curve.  This upper bound will only depend on the generating curve. In 
\cite{PP}, this estimate is modified so that it applies to non circular cylindrical equilibrium surfaces as well. 
\begin{thm}
For a helicoidal surface which is not a round cylinder, a necessary condition for the stability of the part the surface between horizontal planes separated by a distance $h$ is that 
\begin{equation}
\label{bb}
\frac{4\pi^2e^4}{h^2}\ge \frac{\omega^2\oint_\alpha \frac{(1+\omega^2R^2(1+\kappa \xi_2)^2}{(1+\omega^2\xi_1^2)^{7/2}}\:ds}{\oint_\alpha\frac{1}{1+\omega^2\xi_1^2}\:ds}(\ge \omega^2),
\end{equation}
holds. The result also holds true if $a=0$, i.e. if the surface has constant mean curvature.
\end{thm}
{\bf Remark.} In \cite{PP} a similar estimate is given for cylindrical surfaces which are not round cylinders.\\[2mm]
\begin{proof} To begin, note that the the third component of the normal $\nu_3$ satisfies $L[\nu_3]=0$ since vertical translation is a symmetry of the normal.  Also, this function will vanish identically if and only if the surface is a cylinder. 

The function $\nu_3$ can be written \cite{P1}
$$\nu_3=\frac{\omega\xi_1}{\sqrt{1+\omega^2\xi_1^2}}\:,$$
so $\nu_3$ is a function of $s$ only. Using local coordinate expressions found in \cite{P1}, we can write
\begin{eqnarray*}
0&=&L[\nu_3]\\
&=& \frac{1}{\sqrt{g}}(\sqrt{g}g^{11}(\nu_3)_s)_s+(|d\nu|^2+a\xi_2)\nu_3\\
&=& \frac{1}{\sqrt{1+\omega^2\xi_1^2}}\bigl[\frac{1+\omega^2R^2}{\sqrt{1+\omega^2\xi_1^2}}(\nu_3)_s\bigr]_s+(|d\nu|^2+a\xi_2)\nu_3\\
&=:& {\mathcal L}[\nu_3]\:.
\end{eqnarray*}
Note that $(|d\nu|^2+a\xi_2)$ only depends on $s$. For any smooth function $u=u(s)$, there holds
$${\mathcal L} [e^{u}]=e^u({\mathcal L}[u]+g^{11}u_s^2)=e^u({\mathcal L}[u]+(1+\omega R^2)u_s^2)\:.$$

If we now take $\psi=e^{\nu_3(s)}\sin((2\pi t)/h)$, then (\ref{zint})  holds and we can obtain from (\ref{AA}),
$$\delta^2 {\mathcal E}_{a, \Lambda_0}=\frac{2\pi^2}{h}\oint_\alpha \frac{e^{2\nu_3}}{1+\omega^2\xi_1^2}\:ds-
\frac{h}{2}\oint_\alpha \frac{e^{2\nu_3}(1+\omega^2R^2)}{\sqrt{1+\omega^2\xi_1^2}}((\nu_3)_s)^2\:ds$$ 
Using $-1\le \nu_3\le 1$  and using
$$(\nu_3)_s=(\omega \xi_1(1+\omega^2\xi_1^2)^{-1/2})_s=(\xi_1)_s \omega(1+\omega^2\xi_1^2)^{-3/2}
=\omega(1+\kappa \xi_2)(1+\omega^2\xi_1^2)^{-3/2},$$
yields the result. \end{proof}
\section{Appendix}

We assume that $\Sigma$ is contained in a three dimensional region $\Omega$ and that $\partial \Sigma$ is contained in a supporting surface $S$ which is part of
$\partial \Omega$. We assume that there is a (not necessarily connected) domain $S_1\subset S$ such that $\Sigma \cup S_1$ bounds a subregion $\Omega_1\subset \Omega$.
The volume of $\Omega_1$ will be denoted by ${\mathcal V}$.

Let $\phi$ be a solution of $\Delta'\phi=1$ in $\Omega$ with $\nabla'\phi \cdot N=0$ on $S$ where $N$ is the outward pointing normal to $S$. This boundary value problem is undertermined and is solvable provided $S$ is not closed.

We subject the surface to a variation which keeps $\partial \Sigma$ on $S$. Then $\delta X=:T+\psi \nu \perp N$ along $\partial \Sigma$ and
$${\mathcal V}=\int_\Sigma \nabla'\phi \cdot \nu\:d\Sigma\:.$$

We have 
\begin{eqnarray*}
\delta {\mathcal V}&=&\int_\Sigma \nabla '_{T+\psi \nu}\nabla '\phi \cdot \nu+\nabla '\phi \cdot \delta \nu \:d\Sigma +\int_\Sigma \nabla'\phi \cdot \nu(\nabla \cdot T-2H\psi)\:d\Sigma \\
&=& \int_\Sigma \psi \nabla'_\nu \nabla '\phi\cdot \nu -2H\psi \nabla '\phi \cdot \nu -\nabla \phi \cdot \nabla \psi \:d\Sigma +\oint_{\partial \Sigma} (\nabla '\phi \cdot \nu)T\cdot n\:ds \\
&=& \int_\Sigma \psi \nabla'_\nu \nabla '\phi\cdot \nu -2H\psi \nabla '\phi \cdot \nu +\psi \Delta \phi\:d\Sigma\\
&& +\oint_{\partial \Sigma} ((\nabla '\phi \cdot \nu)T-\psi \nabla \phi)\cdot n\:ds \:.
\end{eqnarray*}
A well known formula relating the Laplacian on a submanifold to the Laplacian on the ambient space gives $\Delta '\phi=\nabla'_\nu \nabla '\phi\cdot \nu -2H\nabla '\phi +\Delta \phi$. Therefore we obtain 
$$\delta {\mathcal V}=\int_\Sigma \psi \:d\Sigma +\oint dX\times \nabla'\phi \cdot \delta X\:.$$
However, all of $dX$,$ \nabla'\phi $ and $\delta X$ are perpendicular to $N$ on $\partial \Sigma$, so the line integral above vanishes.

To obtain (\ref{dR}), we let $W$ be a vector field on $\Omega$ satisfying $\nabla '\cdot W  =R^2$ and $W\cdot N=0$ along $S$. This boundary value problem is undertermined and is solvable provided $S$ is not closed. 
 Then by the Divergence Theorem
$$\int_{\Omega_1}R^2\:d^3x=\int_\Sigma W\cdot \nu\:d\Sigma \:.$$
\begin{eqnarray*}
\delta \int_{\Omega_1}R^2\:d^3x&=& \int_\Sigma \nabla '_{T+\psi \nu}W\cdot \nu +W\cdot \delta \nu \:d\Sigma +\int_\Sigma W\cdot \nu \:(\nabla\cdot T-2H\psi)\:d\Sigma\\
 &=&\int_\Sigma \psi \nabla'_\nu W\cdot \nu -W\cdot \nabla \psi-2H\psi W\cdot \nu\:d\Sigma +\oint_{\partial \Sigma} (W\cdot \nu)T\cdot n\:ds\\
 &=&\int_\Sigma \psi \nabla '\cdot W\:d\Sigma +\oint_{\partial \Sigma}( (W\cdot \nu)T-\psi W)\cdot n\:ds\\
 &=&\int_\Sigma \psi R^2\:d\Sigma +\oint_{\partial \Sigma}dX\times W\cdot \delta X\:.
 \end{eqnarray*}

Again, all of $dX$, $W$ and $\delta X$ are perpendicular to $N$ along $\partial \Sigma$ so the line integral will vanish.

If the pair $(\phi,W)$ used above are replaced by another pail $({\underline \phi}, {\underline W})$ satisfying the same equations:
$\Delta' {\underline \phi}=1$ and $\nabla'{\underline W}=R^2$, then by the Divergence Theorem 
$$\int_\Sigma \nabla'{\underline \phi} \cdot \nu\:d\Sigma={\mathcal V}+c_1\:,$$
$$\int_\Sigma W\cdot \nu\:d\Sigma =\int_{\Omega_1}R^2\:d^3x+c_2\:,$$
for constants $c_1$ and $c_2$. Thus these replacements will not affect the variational formulas for these integrals.

\end{document}